%% file: unimodular.tex
\documentclass{style}
\input{macros}

\begin{document}

\begin{abstract}
Let $\C$ be a finite tensor category and $\M$ an exact left $\C$-module category. We call $\M$ unimodular if the finite multitensor category $\Rex_{\C}(\M)$ of right exact $\C$-module endofunctors of $\M$ is unimodular. In this article, we provide various characterizations, properties, and examples of unimodular module categories. 
As our first application, we employ unimodular module categories to construct (commutative) Frobenius algebra objects in the Drinfeld center of any finite tensor category. When $\C$ is a pivotal category, and $\M$ is a unimodular, pivotal left $\C$-module category, the Frobenius algebra objects are symmetric as well. Our second application is a classification of unimodular module categories over the category of finite dimensional representations of a finite dimensional Hopf algebra; this answers a question of Shimizu \cite{shimizu2022Nakayama}. Using this, we provide an example of a finite tensor category whose categorical Morita equivalence class does not contain any unimodular tensor category.

\end{abstract}


\maketitle

\setcounter{tocdepth}{2}

{\small
\tableofcontents
}

\section{Introduction}
Unimodularity is a classical notion, with roots in linear algebra. Building upon the definition of a unimodular matrix (determinant  $=\;\pm 1$), unimodularity of lattices, bilinear forms, topological groups, Hopf algebras, Poisson algebras, tensor categories, etc.~is defined. In this work, we contribute to the literature by defining and studying unimodular module categories. 

\smallskip

Research in this direction began with work on locally compact topological groups. Such groups are equipped with a left and a right invariant Haar measure, and when the left invariant Haar measure is also right invariant, we call the group \textit{unimodular} \cite{hewitt2012abstract}.
Generalizing this, Sweedler \cite{sweedler1969integrals} introduced the notions of left and right integrals for Hopf algebras. Then, a finite dimensional Hopf algebra $H$ is said be unimodular if its \textit{distinguished character}, which measures how far a left integral is from being a right integral, is identically the unit element of $H$ \cite{larson1969associative}. Semisimple Hopf algebras, for instance, are unimodular \cite{larson1969associative}. Etingof and Ostrik \cite{etingof2004finite} defined an analogue of the distinguished character called the \textit{distinguished invertible object}, denoted as $D$, for any tensor category $\C$.
Then, $\C$ is called \textit{unimodular} if $D$ is isomorphic to the unit object of $\C$. 
As one would expect: \\
\textbullet \; the tensor category $\Rep(H)$ is unimodular precisely when $H$ is unimodular; and \\
\textbullet \; semisimple tensor categories are unimodular \cite{etingof2004analogue}. \\
In fact, many results in the semisimple case generalize to the unimodular setting.

\smallskip

Unimodularity of a tensor category is a crucial property for topological applications. 
Non-semisimple generalizations of the Reshetikhin-Turaev invariants \cite{reshetikhin1991invariants}, which are defined using certain tensor categories as input, require the input category to be unimodular. Another important class of invariants, the Turaev-Viro invariants \cite{turaev1992state}, are built using spherical fusion categories as input. However, in the non-semisimple setting, the definition of a spherical tensor category \cite{douglas2018dualizable} requires the underlying category to be unimodular. Recent works like \cite{beliakova2022kerler} have also employed unimodular (ribbon) tensor categories to construct invariants of $4$-dimensional manifolds.

\smallskip

In this work, we define what it means for an exact module category over a finite tensor category to be unimodular. The following definition is inspired by \cite[Remark~4.27]{fuchs2020eilenberg}.

\begin{definition}\label{defn:unimodular}
    An exact left $\C$-module category $\M$ is called \textit{unimodular} if the multitensor category $\Rex_{\C}(\M)$ of right exact, left $\C$-module endofunctors of $\M$ is unimodular.
\end{definition}

Our primary motivation for studying unimodular module categories is to provide a functorial construction of (commutative, special, symmetric) Frobenius algebras in \textit{modular tensor categories} (MTCs). MTCs are important models for the mathematical study of $2$-dimensional rational Conformal Field Theories (CFTs) \cite{fuchs2002tft} and $3$-dimensional Topological Quantum Field Theories (TQFTs) \cite{reshetikhin1991invariants}. 
For a $2$D-CFT modeled by a modular tensor category $\C$, commutative, symmetric Frobenius algebra objects in $\C$ are in bijection with a consistent system of bulk field correlators \cite{fuchs2017consistent}. On the other hand, for $3$D-TQFTs built using a MTC $\C$ \cite{turaev1992modular,kerler2001non,bartlett2015modular}, commutative, special, symmetric Frobenius algebras in $\C$ can be used to produce new MTCs via the construction of the category of local modules \cite{pareigis1995braiding,schauenburg2001monoidal,kirillov2002q,laugwitz2022local}. 

\smallskip

While prior works like \cite{frohlich2006correspondences} have dealt with constructions of Frobenius algebra objects in the case when the MTC is semisimple, recent efforts towards constructing non-semisimple CFTs \cite{fuchs2021bulk} and TQFTs \cite{de20223} require new constructions of Frobenius algebra objects in general MTCs. 
In this work, we address this requirement for MTCs obtained as Drinfeld centers of spherical tensor categories. More generally, we employ unimodular module categories to construct `nice' Frobenius algebra objects in the Drinfeld center $\Z(\C)$ of any finite tensor category. 


\subsection{Main results}
The background material needed for the following discussion can be found in Section~\ref{sec:background}. Let $\C$ be a finite tensor category and $\M$ an exact left $\C$-module category. Our strategy to construct (special, separable, symmetric) Frobenius algebras in $\Z(\C)$ is to construct Frobenius monoidal functors with target $\Z(\C)$. Such functors preserve Frobenius algebras and can be used to transfer (special, separable, symmetric) Frobenius algebras from the input category to $\Z(\C)$.

To construct such functors, we employ unimodular module categories, which are studied in Section~\ref{sec:unimodular}. A major role is played by the following functor introduced in \cite{shimizu2020further}.
\[\Psi:=\Psi_{\M}: \Z(\C) \rightarrow \Rex_{\C}(\M), \hspace{1cm} (c,\sigma)\mapsto (c\tr - , s^{\sigma}).\]
Building on prior works \cite{shimizu2016unimodular, spherical2022,yadav2022frobenius}, we provide various characterizations of unimodularity of module categories as follows. In particular, we obtain the following result.

\begin{namedtheorem}[A]\textnormal{(Theorem~\ref{thm:equivalentUni})}\label{thm:intro2}
Let $\C$ be a finite tensor category and $\M$ an indecomposable, exact, left $\C$-module category. Then, the following are equivalent.
\begin{enumerate}
    \item[\upshape{(a)}] $\M$ is a unimodular module category.
    \item[\upshape{(b)}] $\Rex_{\C}(\M)$ is a unimodular tensor category.
    \item[\upshape{(c)}] $\Psi^{\ra}(\id_{\M})$ is a Frobenius algebra in $\Z(\C)$.
    \item[\upshape{(d)}] $\Psi^{\ra}$ is a Frobenius monoidal functor. 
\end{enumerate}
\end{namedtheorem}
\noindent
Here, for any functor $F$, $F^{\la}$ (resp., $F^{\ra}$) denotes the left (resp., right) adjoint of $F$. 
In particular, part (d) yields a supply of Frobenius algebras in the Drinfeld center. Furthermore, the following result describes when the functor $\Psi^{\ra}$ and the algebra obtained using it are separable or special.

\begin{namedcorollary}[B]\textnormal{(Corollary~\ref{cor:PsiRaSS})}
Let $\C$ be a finite tensor category and $\M$ be an indecomposable, unimodular left $\C$-module category. Then, $\Psi^{\ra}$ is a separable (resp. special) Frobenius monoidal functor if and only if the Frobenius algebra $\Psi^{\ra}(\id_{\M})$ in $\Z(\C)$ is separable (resp. special).    
\end{namedcorollary}

Our goal is to construct special, symmetric Frobenius algebras. However, to discuss symmetric Frobenius algebras, we have to move to the pivotal setting.
When $\C$ is pivotal and $\M$ is a pivotal left $\C$-module category, then the categories $\Z(\C)$ and $\Rex_{\C}(\M)$ are pivotal \cite{schaumann2015pivotal,shimizu2019relative}, and $\Psi$ is a pivotal functor. Then, the following result describes sufficient conditions needed to ensure that the functor $\Psi^{\ra}$ is a pivotal functor. When these conditions are satisfied, $\Psi^{\ra}$ becomes a tool of producing special, symmetric Frobenius algebras in $\Z(\C)$.

\begin{namedtheorem}[C]\textnormal{(Theorems~\ref{thm:PsiPivotal}, \ref{thm:PsiSpecial})}\label{thm:intro3}
Let $\C$ be a pivotal finite tensor category and $\M$ an indecomposable, unimodular, pivotal left $\C$-module category. Then, 
\begin{itemize}
    \item $\Psi^{\ra}$ is a pivotal Frobenius monoidal functor. 
    \item $\Psi^{\ra}$ is a special pivotal Frobenius monoidal functor if and only if $\dim(\Psi^{\ra}(\id_{\M}))\neq 0$.
\end{itemize}
\end{namedtheorem}

\begin{remark}
Given $F\in \Rex_{\C}(\M,\N)$, the functor $F\circ F^{\la}$ is an algebra in $\Rex_{\C}(\N)$. Then, one can show that the algebras $\Psi^{\ra}(F\circ F^{\la})$ are isomorphic to the algebras $\uNat(F,F)$ introduced in \cite{fuchs2021internal}. One of their main results \cite[Corollary~19]{fuchs2021internal} provides sufficient conditions under which the algebras $\uNat(F,F)$ are symmetric Frobenius. This result prompted us to investigate the functor $\Psi^{\ra}$ in this paper. 
\end{remark}


In Section~\ref{sec:comoduleAlg}, we consider the case when $\C:=\Rep(H)$ is the category of finite-dimensional representations of a finite-dimensional Hopf algebra $H$. In this case, every exact left $\C$-module category $\M$ is of the form $\Rep(A)$ for $A$ a left $H$-comodule algebra. We employ results from \cite{shimizu2019relative,shibata2021modified} to calculate the distinguished invertible object $D_{\Rex_{\C}(\M)}$ of $\Rex_{\C}(\M)$. To understand when $\Rex_{\C}(\M)$ is unimodular, that is, when $D_{\Rex_{\C}(\M)}$ and $\id_{\M}$ are isomorphic, we introduce \textit{unimodular elements} of an exact $H$-comodule algebras (Definition~\ref{defn:uniElement}) and obtain the following result.

\begin{namedtheorem}[D]{\textup{(Theorem~\ref{thm:uniclassification})}}
Let $H$ be a finite-dimensional Hopf algebra and $A$ an exact left $H$-comodule algebra. Then the left $\Rep(H)$-module category $\Rep(A)$ is unimodular if and only if $A$ admits a unimodular element. 
\end{namedtheorem}

\noindent
The question of unimodularity of $\M$ (which, by definition, is equivalent to the unimodularity of $\Rex_{\C}(\M)$) has also recently been investigated in \cite{shimizu2022Nakayama} in the case when the algebra $A$ admits a grouplike cointegral. Theorem~D provides explicit conditions for when the module category $\Rep(A)$ is unimodular without any such assumption, thereby answering \cite[Question~7.25]{shimizu2022Nakayama}. Furthermore, Theorem~D reduces to Shimizu's result \cite[Corollary~7.10]{shimizu2022Nakayama}, when the comodule algebra $A$ admits a grouplike cointegral; see Corollary~\ref{cor:grouplikecointegral}.

\smallskip

A natural question to ask is, whether every finite tensor category admits a unimodular module category. We obtain the following result which provides a negative answer to this question. 

\begin{namedtheorem}[E]{\textup{(Theorem~\ref{thm:Taft})}}
Let $T(\omega)$ denote the Taft algebra. Then,
the category $\Rep(T(\omega))$ does not admit a unimodular module category.
\end{namedtheorem}

We end this article with some remarks and questions in Section~\ref{subsec:remarks}.


\subsection{Acknowledgements}
The author would like to thank Chelsea Walton for her guidance and support throughout the project. We also thank Anh-Tuong Nguyen and Kenichi Shimizu for helpful correspondences. The author is partially supported by a Nettie S.~Autrie Research Fellowship from Rice University.


\section{Background on monoidal categories}\label{sec:background}
In this article, we work over an algebraically closed field $\kk$.
We review monoidal categories and module categories over them in Section~\ref{subsec:1mon}, rigid and pivotal categories in Section~\ref{subsec:2pivotal}, finite tensor categories and their module categories in Section~\ref{subsec:5inthom}, algebras in monoidal categories in Section~\ref{subsec:3alg}, and Drinfeld centers in Section~\ref{subsec:9drinfeld}. 
We refer the reader to the textbooks \cite{mac2013categories}, \cite{etingof2016tensor} and \cite{turaev2017monoidal} for further details.

\subsection{Monoidal categories and module categories}\label{subsec:1mon}
\subsubsection{Monoidal categories}
A category $\C$ equipped with a functor $\otimes: \C\times\C\rightarrow \C$ (called the \textit{tensor product}), an object $\unit\in \C$ (called the \textit{unit object}) and natural isomorphisms
\begin{equation}\label{eq:monoidalCat}
    X\otimes(Y \otimes Z) \cong (Z\otimes Y)\otimes Z,\;\; X\otimes \unit \cong X \cong \unit \otimes X , \;\;\;\; 
\text{for all}\; X,Y,Z\in \C,
\end{equation}  
is called a \textit{monoidal category} if the isomorphisms in (\ref{eq:monoidalCat}) satisfy the pentagon and the triangle axioms. If these isomorphisms are identities, we call $\C$ a \textit{strict} monoidal category. By Mac Lane's coherence theorem \cite[VII.2]{mac2013categories}, we can (and will) assume that all monoidal categories are strict.
We let $\C^{\rev}$ denote the category $\C$ with the opposite tensor product $\otimes^{\rev}$, that is, $X\otimes^{\rev}Y :=Y\otimes X$. We denote the opposite category of $\C$ as $\C^{\op}$. Then both $(\C^{\op},\otimes,\unit)$ and $(\C^{\rev},\otimes^{\rev},\unit)$ are monoidal categories. 
 
\smallskip

A \textit{braided monoidal category} is a monoidal category $(\C,\otimes,\unit)$ equipped with a natural isomorphism $c=\{c_{X,Y}:X\otimes Y\rightarrow Y\otimes X \}_{X,Y\in\C}$ (called a \textit{braiding}) satisfying the hexagon axiom. The \textit{mirror} $c'$ of a braiding $c$ on $\C$ is defined by $c'_{X,Y}:=c_{Y,X}^{-1}$. We will let $\C^{\mir}$ denote the braided monoidal category $(\C,\otimes,\unit,c')$. 


\subsubsection{Monoidal functors} Let $(\C,\otimes_{\C},\unit_{\C})$ and $(\D,\otimes_{\D},\unit_{\D})$ be two monoidal categories. 

A \textit{monoidal functor} from $\C$ to $\D$ is a tuple $(F,F_2,F_0)$ consisting of a functor $F:\C\rightarrow \D$, a natural transformation $F_2 = \{F_2(X,Y):F(X)\otimes_{\D} F(Y) \rightarrow F(X\otimes_{\C} Y)  \}_{X,Y\in \C}$ and a morphism $F_0:\unit_{\D} \rightarrow F(\unit_{\C}) $ such that certain compatibility conditions are satisfied. We call a monoidal functor $(F,F_2,F_0)$ \textit{strong} (resp., \textit{strict}) if $F_2$ and $F_0$ are isomorphisms (resp., identity maps) in $\D$. If $F$ is strong monoidal and an equivalence between the underlying categories, we call it a \textit{monoidal equivalence}.

\smallskip

A \textit{comonoidal functor} from $\C$ to $\D$ is a tuple $(F,F^2,F^0)$ consisting of a functor  $F:\C\rightarrow \D$, a natural transformation $F^2 = \{F^2(X,Y): F(X\otimes_{\C} Y)  \rightarrow F(X)\otimes_{\D} F(Y) \}_{X,Y\in \C}$ and a morphism $F^0: F(\unit_{\C}) \rightarrow \unit_{\D} $ such that $(F^{\op},(F^2)^{\op},(F^0)^{\op}): \C^{\op}\rightarrow \D^{\op}$ is a monoidal functor.

\smallskip

A \textit{Frobenius monoidal functor} \cite[Definition~1]{day2008note} is a tuple $(F:\C\rightarrow \D,F_0,F_2,F^2,F^0)$ where $(F,F_2,F_0)$ is a monoidal functor, $(F,F^2,F^0)$ is a comonoidal functor and for all $X,Y,Z \in \C$, the following holds:
\begin{align*}
(\id_{F(X)}\otimes_{\D} F_2(Y,Z) )\;(F^2(X,Y)\otimes_{\D} \id_{F(Z)}) &  = \, F^2(X,Y\otimes_{\C} Z) F_2(X\otimes_{\C} Y,Z), \\
(F_2(X,Y)\otimes_{\D} \id_{F(Z)})\; ( \id_{F(X)}\otimes_{\D} F^2(Y,Z)) & =  \, F^2(X\otimes_{\C} Y, Z) F_2(X, Y \otimes_{\C} Z). 
\end{align*}
If the categories $\C,\,\D$ are $\kk$-linear, then, such a functor $F$ is called \textit{separable} if it satisfies that $F_2(X,Y)\circ F^2(X,Y)=\beta_2\; \id_{F(X\otimes Y)}$ for some $\beta_2\in\kk^{\times}$. If in addition, $F^0\circ F_0 = \beta_0 \; \id_{\unit}$ holds for some $\beta_0\in\kk^{\times}$, we call $F$ \textit{special}.

\subsubsection{$\C$-Module categories}  
Let $\C$ be a monoidal category. A \textit{left $\C$-module category} is a category $\M$ equipped with a functor $\tr:\C\times \M\rightarrow \M$ (called the \textit{action} of $\C$) and natural isomorphisms
\[(X\otimes Y)\tr M \cong X\tr(Y\tr M),\;\; \unit\tr M \cong M\;\;\;\;\;\; \text{for all} \;X,Y\in \C, M\in \M \] 
satisfying certain coherence conditions.
Analogously, one can define right $\C$-module categories and bimodule categories.
By a variant of Mac Lane's coherence theorem (see \cite[Remark~7.2.4]{etingof2016tensor}), we can (and will) assume that the above isomorphisms are identity maps.

\subsubsection{$\C$-module functors}
Let $(\M,\tr_{\M})$ and $(\N,\tr_{\N})$ be left $\C$-module categories. A \textit{left $\C$-module functor} is a tuple $(F,s)$ where:
\begin{itemize}
    \item $F:\M\rightarrow\N$ is a functor, and
    \item $s= \{s_{X,M}: F(X\tr_{\M} M)  \rightarrow  X\tr_{\N} F(M) \}_{X\in\C,M\in\M}$ is a natural isomorphism satisfying
    \[ s_{X\otimes Y,M} = (\id_X \tr_{\N} s_{Y,M})\circ  \; s_{X,Y\tr M},  \;\;\;\;\; s_{\unit,M} = \id_{F(M)}  \hspace{1.5cm} (X,Y\in\C, M\in \M).\]
\end{itemize}
One can define right $\C$-module functors and bimodule functors in a similar manner.

\begin{example}[${}_F\M,\tr_F$]\label{ex:twisted}
Let $\C,\,\D$ be monoidal categories and $(\M,\tr)$ a left $\D$-module category. Given a strong monoidal functor $F:\C\rightarrow\D$, we will denote by $({}_F\M,\tr_F)$ the category $\M$ with $\C$-action given by
    \begin{equation*}
        X \tr_F M: =F(X)\tr M , \hspace{1cm} (X\in \C,M\in \M).
    \end{equation*} 
    Then $({}_F\M,\tr_F)$ is a left $\C$-module category. In this case, we call $\M$ a {\it ($F$-)twisted left $\C$-module category}. 
\end{example}

\medskip
Let $(F,s^F),(G,s^G): (\M,\tr_{\M}) \rightarrow (\N,\tr_{\N})$ be left $\C$-module functors. A \textit{left $\C$-module natural transformation} is a natural transformation $\eta: F\rightarrow G$ satisfying
\begin{equation*}\label{eq:modulenattrans}
    (\id_X \tr_{\N} \eta_M) \circ s^F_{X,M} = s^G_{X,M} \circ \eta_{X\tr_{\M} M}    \hspace{1.5cm} (X\in\C, M\in \M).
\end{equation*}


\subsection{Duality in monoidal categories}\label{subsec:2pivotal}
A monoidal category is called \textit{rigid} if every object $X$ in $\C$ comes equipped with a left and right dual, i.e., there exist an object $\lv X$ (\textit{left dual}) along with co/evaluation maps $\ev_X:\lv X\otimes X\rightarrow \unit$, $\coev_X:\unit\rightarrow X\otimes \lv X$ and an object $X\rv$ (\textit{right dual}) with co/evaluation maps $\widetilde{\ev}_X:X\otimes X\rv\rightarrow \unit$, $\widetilde{\coev}_X: \unit\rightarrow X\rv \otimes X$ satisfying the usual snake relations.

\smallskip

The maps $X\mapsto \lv X$ and $X \mapsto X\rv$ extend to monoidal equivalences from $\C^{\rev}$ to $\C^{\op}$ . 
We can and will replace $\C$ by an equivalent monoidal category and choose duals in a suitable way to ensure that $\lv(-)$ and $(-)\rv$ are strict monoidal and mutually inverse to each other, \cite[Lemma~5.4]{shimizu2015pivotal}. 


\subsubsection{Pivotal categories}

A monoidal category with left duals is called \textit{pivotal} if it comes equipped with a natural isomorphism $\fp=\{\fp_X:X\rightarrow \lv\lv X\}_{X\in\C}$ satisfying
$\fp_{X\otimes Y} = \fp_X\otimes \fp_Y$.
In a pivotal category, for each object $X$, $\lv X$ is also a right dual to $X$ with co/evaluation maps 
\begin{equation*}\label{eq:pivtilde}
    \widetilde{\coev}_X := (\id_{\lv X}\otimes \fp_X^{-1})\coev_{\lv X} \;\;\;\; \text{and} \;\;\;\; \widetilde{\ev}_X := \ev_{\lv X} (\fp_X \otimes \id_{\lv X}).
\end{equation*}
The (quantum) dimension of an object $X\in\C$ (with respect to a pivotal structure $\fp$) is defined as the following endomorphism of the unit object.
\begin{equation*}\label{eq:qdim}
    \dim(X):=\dim_{\C}^{\fp}(X) = \tev_X\circ \coev_X \in \End(\unit_{\C}).
\end{equation*}
If $\C$ is $\kk$-linear and $\End(\unit_{\C})\cong\kk$, then $\dim(X)$ is a scalar.


\subsubsection{Pivotal Frobenius functors}
Suppose that we have a Frobenius monoidal functor $(F,F_2,F_0,F^2,F^0)$ between rigid monoidal categories $\C,\; \D$. Then by \cite[Theorem~2]{day2008note}, $F(\lv X)$ is a left dual of $F(X)$ for $X\in\C$.
Thus, by uniqueness of dual objects, we get a unique family of natural isomorphisms $\zeta^F_X: F(\lv X) \rightarrow \lv F(X)$. Generalizing the definition of pivotal functors from \cite{ng2007higher} gives the following definition of pivotal Frobenius functors.

\begin{definition}
A Frobenius monoidal functor $F:\C\rightarrow\D$ between pivotal categories is said to be \textit{pivotal} if it satisfies 
$\lv(\zeta^F_X) \circ \fp^{\D}_{F(X)} = \zeta^F_{\lv X} \circ F(\fp^{\C}_X)$ for all  $X\in\C$. 
\end{definition}

\begin{lemma}\label{lem:pivShimizuB}\cite[Lemma~2.7]{yadav2022frobenius}
Let $\C\xrightarrow{F}\D\xrightarrow{G}\E$ be a sequence of Frobenius monoidal functors between rigid monoidal categories. If $\C,\; \D, \; \E$ are pivotal and $F,\; G $ are pivotal functors, then so is their composition $G\circ F$.  \qed
\end{lemma}

Next, we see how the dimension of an object $X$ in $\C$ is related to the dimension of $F(X)$ in $\D$ when $F$ is a pivotal Frobenius monoidal functor.
\begin{lemma}\label{lem:dimFrob0}
Let $F:(\C,\fp)\rightarrow(\D,\fq)$ be a pivotal Frobenius monoidal functor. Then 
\begin{equation}\label{eq:dimFrob}
    \dim^{\fq}_{\D}(F(X)) = F^0\; F(\tev_X)\;  F_2(X, \lv X) \; F^2(X,\lv X) \; F(\coev_X) \; F_0 .
\end{equation}
\end{lemma}
\begin{proof}
This follows from a straightforward calculation using the definition of $\zeta^F_X$.
\end{proof}

\begin{lemma} \label{lem:dimPiv}
Let $(F,F_2,F_0,F^2,F^0):\C\rightarrow\D$ be a pivotal, special, Frobenius monoidal functor (with constants $\beta_0,\beta_2$) between pivotal monoidal categories $(\C,\fp)$ and $(\D,\fq)$. If $\End(\unit_{\C})\cong \kk$, then,
\begin{equation*}
    \dim^{\fq}_{\D}(F(X)) = \beta_2 \beta_0 \dim^{\fp}_{\C}(X)\, \id_{\unit_{\D}}.
\end{equation*}
\end{lemma}
\begin{proof}
The following calculation proves the result. 
\begin{align*}
\dim^{\fq}_{\D}(F(X))  \stackrel{(\textnormal{\ref{eq:dimFrob}})}{=} & F^0\; F(\tev_X)\;  F_2(X, \lv X) \; F^2(X,\lv X) \; F(\coev_X) \; F_0 \\
= \;\;& F^0 \; F( \tev_X) \; (\beta_2 \id_{F(X\otimes \lv X)}) \; F(\coev_X) \; F_0  \\
= \;\;& \beta_2 \; F^0 \; F( \tev_X \circ  \coev_X) \; F_0 \\
= \;\;& \beta_2 \dim^{\fp}_{\C}(X)\; F^0\; F_0 \, \id_{\unit_{\D}} \\
= \;\;& \beta_2 \beta_0 \; \dim^{\fp}_{\C}(X)\, \id_{\unit_{\D}}. \qedhere  
\end{align*}
\end{proof}


\subsection{Finite multitensor categories and their module categories}\label{subsec:5inthom}
For a $\kk$-algebra $A$, let $\Rep(A)$ denote the category of finite dimensional left $A$-modules over $\kk$. A \textit{finite abelian category} is a $\kk$-linear category that is equivalent to $\Rep(A)$ for some finite dimensional $\kk$-algebra $A$.

\smallskip

A \textit{finite multitensor category} is a rigid monoidal category that is finite abelian and the tensor product functor $\otimes$ is $\kk$-linear in each variable. If further the unit object $\unit$ is simple, we call it a \textit{finite tensor category}. 
A \textit{tensor functor} is an exact, faithful, $\kk$-linear, strong monoidal functor between finite multitensor categories. 

\smallskip

Let $\C$ be a finite multitensor category. A \textit{finite left $\C$-module category} is a left $\C$-module category $\M$ such that $\M$ is a finite abelian category and the action of $\C$ on $\M$ is $\kk$-bilinear and right exact in each variable. 
Such a category is called \textit{exact} if for all objects $M\in \M$ and all projective objects $P\in \C$, $P\tr M$ is projective in $\M$. A module category is called \textit{indecomposable} if it is not equivalent to a direct sum of two non-trivial module categories.

\begin{notation}
We will use the following notations going forward.
\begin{enumerate}
\item For $\M,\N$ two finite left $\C$-module categories, $\Rex_{\C}(\M,\N)$ will denote the category of right exact $\C$-module functors from $\M$ to $\N$. If $\M=\N$, we call it $\Rex_{\C}(\M)$. 
\item Given $(F,s^F),(G,s^G)$ two left $\C$-module functors, we will use the notation $F\cong_{\C}G$ to mean that $F$ and $G$ are isomorphic as left $\C$-module functors. 
\item For $X\in \C$ and $F\in \Rex_{\C}(\M,\N)$, we use the notation $X\tr F$ to denote the functor from $\M$ to $\N$ defined as follows: $(X\tr F)(M)=X\tr F(M)$.
\end{enumerate}
\end{notation}

\subsubsection{Internal Homs}
We refer the reader to \cite[Section~7.4]{etingof2016tensor} for a more detailed exposition. Let $\C$ be a monoidal category and $(\M, \tr)$ be a left $\C$-module category. Consider the functor
\begin{equation*}
    Y_M:\C\rightarrow \M, \hspace{1cm} X\mapsto X\tr M.
\end{equation*}
If $Y_M$ admits a right adjoint, we will denote it by $\uHom(M,-):=\uHom^{\C}_{\M}(M,-):\M\rightarrow \C$ and call $\uHom(M,N)$ as the \textit{internal Hom} of $M$ and $N$. 
In this case, by definition of adjoint functors, we get the following isomorphism of hom spaces:
\begin{equation}\label{eq:internalHom}
    \Hom_{\M}(X\tr M,N) \cong \Hom_{\C}(X,\uHom(M,N))
\end{equation}
For $\C$ a finite multitensor category and $\M$ a finite left $\C$-module category, the functor $Y_M$ is a right exact functor between finite abelian categories. Hence, it admits a right adjoint, that is, internal Homs exist. 
In fact, the internal Hom extends to a functor $\uHom: \M^{\op}\times \M \rightarrow \C$ in such a way that the isomorphism (\ref{eq:internalHom}) is natural in both $M$ and $N$ \cite[\S2.4]{shimizu2020further}.


\subsubsection{Canonical $\Vect$ action}
Let $\Vect$ denote the category of finite dimensional $\kk$-vector spaces.
Every finite abelian category $\M$ comes equipped with a canonical action of $\Vect$ given by $\btr: \Vect\times \M \rightarrow \M$ defined via the following isomorphism. 
\begin{equation*}
    \Hom_{\M}(V\btr M,N)\cong \Homk(V,\Hom_{\M}(M,N))
\end{equation*} 
With this action, every finite abelian category becomes a left $\Vect$-module category. From here on, we will use the notation $\btr$ to denote this action of $\Vect$ on any finite abelian category.


\subsubsection{Relative Serre functors}\label{subsec:Serre}
Let $\C$ be a finite multitensor category and $\M$ a finite  left $\C$-module category. A \textit{relative (right) Serre functor} \cite[Definition~4.22]{fuchs2020eilenberg} of $\M$ is a pair $(\dS,\phi)$ where $\dS:=\dS^{\C}_{\M}:\M\rightarrow \M $ is a functor and 
\begin{equation*}\label{eq:SerreFunctor}
    \phi:=\phi^r =\{\phi^r_{M,M'}: \lv \uHom(M,M')\rightarrow \uHom(M',\dS(M))  \}_{M,M'\in\M}
\end{equation*}
is a natural isomorphism. If we want to emphasize the categories $\C,\M$, we will write $\dS^{\C}_{\M}$ instead of $\dS$.
Similarly, a \textit{relative (left) Serre functor} is an endofunctor $\overline{\dS}$ of $\M$ together with a natural isomorphism 
\[\phi^l=\{\phi^l_{M,M'}:  \uHom(M,M')\lv \rightarrow \uHom(\overline{\dS}(M'),M)\}_{M,M'\in\M}.\] 
We refer the reader to the works \cite{schaumann2015pivotal,fuchs2020eilenberg,shimizu2019relative} for further details.
Below, we recall some important results about Serre functors.

\begin{theorem}\cite[Lemma~4.23, Proposition~4.24]{fuchs2020eilenberg}\label{thm:Serre}
Let $\C$ be a finite tensor category and $\M$ a finite left $\C$-module category.
\begin{enumerate}
\item[\upshape(a)] $ \M$ is exact if and only if a relative Serre functor $\dS$ exists.
\item[\upshape(b)] If $\M$ is exact, $\dS$ is a category equivalence with quasi-inverse $\overline{\dS}$.
\item[\upshape(c)] If $\M$ is exact, $\dS$ is a $\C$-module functor from $\M\rightarrow {}_{\lv\lv(-)}\M$, that is, we have natural isomorphisms $\fs_{X,M}: \dS(X\tr M) \xrightarrow{\sim}  \lv \lv X \tr \dS(M) $ satisfying the module compatibility conditions. \qed
\end{enumerate}
\end{theorem}


\subsubsection{Nakayama functors}
We assume that the reader is familiar with ends and coends, see \cite[IX]{mac2013categories} for details.
For $\M$ a finite abelian category, define the \textit{left, right Nakayama functors} \cite[Definition~3.14]{fuchs2020eilenberg} to be the endofunctors $\overline{\dN},\dN:\M\rightarrow \M$, respectively, given by
\[ \overline{\dN}_{\M} (M) = \int_{M'\in \M} \Hom_{\M}(M',M) \btr M', \hspace{1cm}\dN_{\M} (M) =\int^{M'\in \M} \Hom_{\M}(M,M')^* \;\btr M'\]
If the category $\M$ is clear from the context, we will often drop it from the subscript. Below we collect some properties of the Nakayama functor that will be needed later.

\begin{theorem}\cite[Theorem~4.5, Corollary~4.7]{fuchs2020eilenberg}\label{thm:Nak}
Let $\C$ be a finite multitensor category and $\M$ a finite left $\C$-module category. Then, 
\begin{enumerate}
    \item[\upshape{(a)}] $\dN$ is a left $\C$-module functor from $\M\rightarrow{}_{(-)\rv\rv}\M$, that is, we have natural isomorphisms $\fn_{X,M}: \dN(X\tr M)  \xrightarrow{\sim}  X\rv\rv \tr \dN(M) $ satisfying the module compatibility conditions. 
    \vspace{0.05cm}
    \item[\upshape{(b)}] If $\M$ is exact, then $\dN$ is a category equivalence with quasi-inverse $\overline{\dN}$. \qed
\end{enumerate} 
\end{theorem}

\begin{remark}
Theorems~\ref{thm:Serre} and \ref{thm:Nak}  were proved in \cite{fuchs2020eilenberg} under the assumption that $\C$ is a finite tensor category. However, the same arguments work for multitensor categories.
\end{remark}

\subsubsection{Unimodular multitensor categories}
Let $\C$ be a finite multitensor category over $\kk$. 

\begin{definition}
The \textit{distinguished invertible object} of $\C$, denoted $D_{\C}$, is defined as $\overline{\dN}_{\C}(\unit_{\C})$. We call $\C$ \textit{unimodular} if $D_{\C}\cong \unit$.
\end{definition}

\begin{remark}
By assumption, $\kk$ is algebraically closed, and hence perfect. Thus, we can employ results from  \cite{shimizu2016unimodular} about the distinguished invertible object. In particular, by \cite[Lemma~5.1]{shimizu2016unimodular}, the above definition of the object $D_{\C}$ matches the one given in \cite{etingof2004analogue,etingof2016tensor}.
\end{remark}

As $\C$ is a multitensor category, following \cite[\S4.3]{etingof2004analogue}, we can write $\unit=\oplus_{i\in I} \unit_i$ for some finite set $I$, where the objects $\unit_i$ are simple and pairwise non-isomorphic. The categories $\C_{ij}:=\unit_i\otimes \C\otimes \unit_j$ are called as the \textit{component subcategories} of $\C$ and we have that $\C=\oplus_{i,j\in I} \C_{ij}$. Here, for all $i\in I$, $\C_{ii}$ is a finite tensor category with unit object $\unit_i$. Now, consider the following result.



\begin{lemma}\label{lem:uniMulti}
    We have that $\C$ is unimodular if and only if all its component subcategories $\C_{ii}$ are unimodular.
\end{lemma}
\begin{proof}
By \cite[Theorem~5.3]{shimizu2016unimodular}, $D_{\C}=\oplus_{i\in I}D_{\C_{ii}}$. If each $\C_{ii}$ is unimodular, then $\D_{\C_{ii}}\cong\unit_i$. Hence, 
\[\unit_{\C}=\oplus_{i\in I} \unit_i = \oplus_{i\in I} D_{\C_{ii}} \cong D_{\C}. \]
Thus, $\C$ is unimodular. 
Conversely, if $\C$ is unimodular, then we get that $\oplus_{i\in I} \unit_i = \oplus_{i\in I} D_{\C_{ii}}$. Tensoring both sides with $\unit_j$, we get that $D_{\C_{jj}}\cong \unit_j$. Hence, each one of the categories $\C_{ii}$ is unimodular.
\end{proof}

Every finite tensor category $\C$ comes equipped with a natural isomorphism 
\begin{equation}\label{eq:Radford}
    \fR=\{\fR_X : \lv\lv X\otimes D \rightarrow D \otimes X\rv\rv \}_{X\in\C} 
\end{equation}
called the \textit{Radford isomorphism} of $\C$ \cite{etingof2004analogue}. The following results provides a relation between the relative Serre functor and the Nakayama functor of an exact $\C$-module category $\M$.

\begin{theorem}\cite[Theorem~4.26]{fuchs2020eilenberg}\label{thm:NakSerre}
Let $\C$ be a finite tensor category and $\M$ be an exact left $\C$-module category. Then, $\dS_{\M}\cong_{\C} D\tr \dN_{\M}$ as a twisted $\C$-module functor and its $\C$-module constraints are given by 
\begin{equation*}\label{eq:NakModuleConstraint}
    \Scale[0.99]{\fs_{X,M} :=\big( \, D\tr\dN(X\tr M) \xrightarrow{\id\, \tr\, \fn_{X,M}} D\tr( X \rv\rv \tr \dN(M)) \xrightarrow{\fR_X^{-1}\, \tr \,\id}  \lv\lv X \tr (D \tr \dN(M))\, \big)_{X\in\C, M\in\M}} . \qed
\end{equation*}
\end{theorem}


\subsection{Algebras in monoidal categories}\label{subsec:3alg}
In a monoidal category $\C$, we can define algebras and coalgebras \cite[Chapter~6.1]{turaev2017monoidal}.
Furthermore, algebras (coalgebras) and their morphisms in $\C$ form a category, which we denote as $\textsf{Alg}(\C)$ ($\textsf{Coalg}(\C)$).
A \textit{Frobenius algebra} in $\C$ is a 5-tuple $(A,m,u,\Delta,\nu)$ such that $(A,m,u) \in {\sf Alg}(\C)$ and $(A,\Delta,\nu) \in {\sf Coalg}(\C)$ satisfying
\begin{equation*}
    (m \otimes \text{id}_A)(\text{id}_A \otimes \Delta) = \Delta m = (\text{id}_A \otimes m)(\Delta \otimes \text{id}_A).
\end{equation*}
A \textit{morphism} of Frobenius algebras $f:A\rightarrow B$ is a map  in ${\sf Alg}(\C)$ and in ${\sf Coalg}(\C)$.
Frobenius algebras and their morphisms in $\C$ form a category, which we denote by ${\sf Frob}(\C)$.

\smallskip

Now, let $\C$ be a $\kk$-linear monoidal category and $A$ a Frobenius algebra in it.
If $m\circ \Delta = \beta_A \id_A$ holds for some $\beta_A \in\kkx$, we call the Frobenius algebra \textit{separable}. If furthermore, $\nu\circ  u = \beta_{\unit} \id_{\unit}$ holds for some $\beta_{\unit}\in \kkx$, we call it \textit{special}.
A Frobenius algebra $(A,m,u,\Delta,\nu)$ in a pivotal monoidal category $\C$ is called \textit{symmetric} \cite[Definition~3.4]{fuchs2002tft} if it satisfies 
\begin{equation*}\label{eq:symAlg}
    (\nu \; m \otimes \id_{\lv A})(\id_A\otimes\coev_A) =   (\fp_{A\rv} \otimes \nu \; m)(\widetilde{\coev}_A\otimes \id_A) : A\rightarrow \lv A. 
\end{equation*}

We call an algebra $A$ in \textit{connected} if it satisfies $\Hom_{\C}(\unit, A)\cong \kk$.
An algebra $A$ in a braided category $(\C,c)$ is called \textit{commutative} if $m\; c_{A,A}=m$ holds.

\begin{lemma}\label{lem:speFrob}
Let $(A,m,u,\Delta,\nu)$ be a connected, Frobenius algebra in a $\kk$-linear pivotal category $\C$. If $\dim(A)\neq 0$, then $A$ is a special Frobenius algebra in $\C$.
\end{lemma}
\begin{proof}
As Frobenius algebras are self dual, we have that $\dim(A) = \varepsilon m \Delta u$. Also, $m,\Delta$ are maps of left $A$-modules in $\C$. Thus, $m\Delta \in \Hom_A(A,A) \cong \Hom_{\C}(\unit,A) \cong \kk$. If $m\Delta=0$, then we will have that $\dim(A)=0$. Therefore, we get that $m\Delta = \beta_2\id_A$ for some $\beta_2\in \kkx$. Hence, $\dim(A) = \beta_{2}\varepsilon u$. Since $\dim(A)\neq 0$, we have that $\varepsilon u = \beta_0 \id_{\unit}$ for some $\beta_0\in \kkx$. Thus, $A$ is a special Frobenius algebra.
\end{proof}

Next, we recall from \cite{waltonfiltered} an alternate characterization of Frobenius algebras in multitensor categories as we will soon need it. Let $\C$ be a multitensor category and $A$ an algebra in $\C$. Then, a \textit{left ideal} $I$ of $A$ is a tuple $(I,\lambda,\phi)$ where $\lambda:A\otimes I\rightarrow I$ is morphism that satisfies the relations $\lambda(m\otimes\id_I) = \lambda(\id_A\otimes \lambda),\; \lambda(u\otimes\id_I)=\id_I$
and $\phi:I\rightarrow A$ is a monic map satisfying $\phi \lambda_I = m(\id_A\otimes \lambda)$.

\begin{theorem}\cite[Theorem~5.3]{waltonfiltered}\label{thm:walton}
Let $\C$ be an abelian rigid monoidal category.
An algebra $(A,m,u)$ is a Frobenius algebra in $\C$ if and only if there exists a morphism $\nu:A\rightarrow \unit$ so that if a left or right ideal $(I,\lambda,\phi)$ of $A$ factors through $\textnormal{ker}(\nu)$, then $\phi$ is a zero morphism in $\C$. \qed
\end{theorem}

Now, let $\C$ be a multitensor category.
Recall that a multitensor category is additive, and hence it admits direct sums of objects. This means for any two objects $X_1,X_2\in \C$, there exists an object $X_1\oplus X_2\in \C$ along with maps $\iota_k:X_k \rightarrow X_1\oplus X_2$ and $p_k:X_1\oplus X_2\rightarrow X_k$ where $k\in \{1,2\}$ such that the following relations hold, 
\begin{equation*}
    \iota_1 p_1+ \iota_2 p_2 = \id_{X_1\oplus X_2}  \;\;\text{and} \;\;  p_j \iota_k = \delta_{j,k} \id_{X_k} \;\text{for}\; j,k\in \{1,2\}. 
\end{equation*}

It is well known that if $(A_i,m_i,u_i)$ $(i\in{1,2})$ are algebras in $\C$, then so is $(A_1\oplus A_2, m_{12},u_{12})$ with 
\begin{equation*}
    m_{12}= \sum_{k\in\{1,2\}} \iota_k m_k (p_k \otimes p_k) \;\; \text{and} \;\; 
    u_{12} = \sum_{k\in\{1,2\}}  \iota_k u_k 
\end{equation*}
Now consider the following result which strengthens \cite[Proposition~3.21]{fuchs2002tft}.

\begin{proposition}\label{prop:frobDirectSum}
Let $\{(A_i,m_i,u_i) \}_{i\in I}$ be a finite set of algebras in a multitensor category $\C$. 
Then, $A=\oplus_{i\in I}A_i$ is a Frobenius algebra in $\C$ if and only if each $A_i$ is a Frobenius algebra in $\C$. 
\end{proposition}
\begin{proof}
It suffices to prove the result when $|I|=2$, that is, when $A=A_1 \oplus A_2$.

($\Leftarrow$) If $(A_k,m_k,u_k,\Delta_k,\nu_k)\in \Frob(\C)$ for $k\in\{1,2\}$, then by \cite[Proposition~3.21]{fuchs2002tft}, we have that the algebra $(A_1\oplus A_2,m_{12},u_{12},\Delta_{12},\nu_{12})\in \Frob(\C)$ with 
\begin{equation*}
    \Delta_{12} = \sum_{k\in\{1,2\}} (\iota_k \otimes \iota_k)\Delta_k p_k \;\; \text{and} \;\; \nu_{12} = \sum_{k\in\{1,2\}} \nu_k p_k .
\end{equation*}

($\Rightarrow$)
By assumption, $(A=A_1\oplus A_2,m_{12},u_{12})$ is a Frobenius algebra in $\C$. Then, by Theorem~\ref{thm:walton}, there exists a morphism $\nu : A\rightarrow \unit$ such that $\textnormal{ker}(\nu)$ contains no left or right ideal of $A$. Set $\nu_k=\nu\circ \iota_k$ for $k\in\{1,2\}$. In order to prove that $A_1$ and $A_2$ are Frobenius, we will show that $\textnormal{ker}(\nu_k)$ contains no left or right ideal of $A_k$ for $k\in\{1,2\}$.

Suppose that the algebra $A_1$ admits a left ideal $(I, \lambda_1, \phi_1)$ that factors through $\textnormal{ker}(\nu_1)$. Let $\iota:\textnormal{ker}(\nu_1)\rightarrow A_1$ denote the inclusion map. Then, there exists a map $f:I\rightarrow \textnormal{ker}(\nu_1)$ such that the lower triangle ($\dagger$) in the following diagram commutes. 
\begin{equation*}\label{eq:frobdirectsum}
\xymatrix{
    & \textnormal{ker}(\nu) \ar[rd] & & \\
    \textnormal{ker}(\nu_1) \ar@{}[urr]|-(0.4){(\ddagger)} \ar[r]^-{\iota} \ar@{.>}[ru]^{g}  & A_1 \ar[r]^-{\iota_1} & A_1\oplus A_2 \ar[r]^-{\nu} & \unit \\
    I \ar[u]^{f} \ar[ru]_{\phi} \ar@{}[uur]|-(0.25){(\dagger)} & & &  
}
\end{equation*}
As the composition $\nu\,\iota_1\,\iota$ is equal to the zero map, by the universal property of the kernel ker($\nu$), there exists a unique map $g$ making the diagram $(\ddagger)$ commute. 
Now consider the tuple $(I,(\iota_1\otimes \id_I)\lambda, \iota_1\phi)$. As both $\iota_1$ and $\phi$ are monic, their composition $\iota_1\,\phi$ is monic as well. A straightforward check shows that $(I,(\iota_1\otimes \id_I)\lambda, \iota_1\phi)$ is a left ideal of $A_1\oplus A_2$ that factors through ker$(\nu)$. This contradicts the assumption that $A_1\oplus A_2$ is Frobenius by Theorem~\ref{thm:walton}. 

In a similar manner we can show that $A_1$ does not admit a right ideal that factors through $\textnormal{ker}(\nu_1)$. Repeating the same argument for $A_2$ we conclude that both $A_1$ and $A_2$ are Frobenius.
\end{proof}


\subsection{Drinfeld centers}\label{subsec:9drinfeld}
Let $\C$ be a monoidal category. Then the \textit{Drinfeld center} of $\C$, denoted $\Z(\C)$, is defined as the category with objects as pairs $(X,\sigma)$, where $X$ is an object in $\C$, and $\sigma=\{\sigma_Y:Y\otimes X\rightarrow X\otimes Y\}_{Y\in \C}$ is a natural isomorphism (called a \textit{half-braiding}) satisfying 
\[\sigma_{Y\otimes Z} = (\sigma_Y\otimes \id_Z)(\id_Y \otimes  \sigma_Z) ,\hspace{1cm} (Y,Z\in \C).\]
Morphisms $(X,\sigma)\rightarrow (Y,\sigma')$ are given by $f\in \Hom_{\C}(X,Y)$ satisfying $(f\otimes \id_Z)\sigma_Z = \sigma'_Z(\id_Z\otimes f)$. The monoidal product is $(X,\sigma)\otimes (Y,\sigma') = (X\otimes Y, \gamma$) where $\gamma_Z:= (\id_X\otimes \sigma'_Z) (\sigma_Z\otimes\id_Y)$. We have the forgetful functor
\begin{equation*}
    U_{\C}:\Z(\C)\rightarrow \C, \hspace{1cm} (X,\sigma) \mapsto X ,
\end{equation*}
which is strong monoidal. When $\C$ is pivotal, we can equip $\Z(\C)$ with a pivotal structure making $U_{\C}$ a pivotal functor \cite[Exercise~7.13.6]{etingof2016tensor}. Unless stated otherwise, we will consider only this pivotal structure on $\Z(\C)$.

Drinfeld centers are important because they are braided monoidal categories with braiding $c_{(X,\sigma),(Y,\sigma')}:=\sigma'_X$. 
Furthermore, if $\C$ is a (finite) tensor category, then $\Z(\C)$ is a braided (finite) tensor category and $U_{\C}$ is a $\kk$-linear, exact functor.
The map $(X,\sigma)\mapsto (X,\sigma^{-1}): \Z(\C^{\rev})\rightarrow \Z(\C) $ is a canonical equivalence of monoidal categories \cite[Exercise~7.13.5]{etingof2016tensor}. If we take the braidings into account as well, then this is a braided equivalence $ \Z(\C^{\rev}) \rightarrow \Z(\C)^{\mir}$ \cite[Exercise~8.5.2]{etingof2016tensor}.

Let $\C$ be a finite tensor category. Then, the forgetful functor $U_{\C}$ admits a right adjoint $R_{\C}$, which is monoidal. 
The following theorem collects important known results which we will need later. 

\begin{theorem}\label{thm:adjAlg}
Let $\C$ be a finite tensor category. Then, we get the following results.
\begin{enumerate}
    \item[\upshape{(a)}] The algebra $R_{\C}(\unit)$ in $\Z(\C)$ is commutative.
    \item[\upshape{(b)}] The algebra $R_{\C}(\unit)$ is a Frobenius algebra in $\Z(\C)$ if and only if $\C$ is unimodular.
    \item[\upshape{(c)}] If $\C$ is pivotal and unimodular, then $R_{\C}(\unit)$ is a symmetric Frobenius algebra in $\Z(\C)$. 
\end{enumerate}
\end{theorem}
\begin{proof}
Parts (a) follow from \cite[Proposition~6.1]{bruguieres2011exact} and part (b) from \cite[Theorem~5.6(3)]{shimizu2016unimodular}.
Part (c) is \cite[Lemma~7.1]{schweigert2022homotopy}. 
\end{proof}


\section{Unimodular module categories}\label{sec:unimodular}
This section is devoted to studying unimodular module categories. We  provide some basic properties and collect various characterizations of unimodular module categories in Section~\ref{subsec:uniDefn}. In Section~\ref{subsec:frobUni}, we employ unimodular module categories to construct Frobenius algebras in the Drinfeld center. In Section~\ref{subsec:pivUni}, we construct symmetric Frobenius algebras using pivotal categories.

\subsection{Definition and basic properties}\label{subsec:uniDefn}
Let $\C$ be a finite tensor category and $\M$ an exact left $\C$-module category. Recall that $\dS$ and $\dN$ are left $\C$-module functors with module constraints $\fs$ and $\fn$ given in Theorem~\ref{thm:Serre}(c) and Theorem~\ref{thm:Nak}(a), respectively. 
Thus, $\dS\,\dN$ is left $\C$-module endofunctor of $\M$ with module constraints
\begin{equation*}\label{eq:NSmoduleConstriant}
    \fd_{X,M}:\dS\,\dN(X\tr M) \xrightarrow{\dS(\fn_{X,M})}\dS(X\rv\rv\tr \dN(M)) \xrightarrow{\fs_{X\rv\rv,\dN(M)}} X\tr \dS\,\dN(M) 
\end{equation*}
Similarly, $\overline{\dN}\, \overline{\dS} \in \Rex_{\C}(\M)$. Using Theorems~\ref{thm:Nak}(b) and \ref{thm:Serre}(b), we get that $\dS\,\dN = (\overline{\dN}\,\overline{\dS})^{-1}$ in $\Rex_{\C}(\M)$. It was shown in \cite[Proposition~4.13]{spherical2022} (see also \cite[Corollary~6.13]{shimizu2022Nakayama}) that $D_{\Rex_{\C}(\M)} \, \cong_{\C} \,  \overline{\dN}\,\overline{\dS}$. Thus, 
\begin{equation}\label{eq:uniModCat}
    \Rex_{\C}(\M) \text{ is unimodular } \iff  \overline{\dN}\,\overline{\dS} \, \cong_{\C}\, \id_{\M} \iff \dS\,\dN\, \cong_{\C}\, \id_{\M}.
\end{equation}  
The above discussion motivates the following definition.

\begin{definition}\label{defn:uniModCat}
A \textit{unimodular structure} on an exact left $\C$-module category $\M$ is a $\C$-module natural isomorphism $\fu: \id_{\M} \rightarrow \dS\;\dN$, namely an isomorphism $\fu:\id_{\M} \rightarrow \dS\;\dN$ of functors such that (\ref{eq:unimodularModule}) commutes for all $X\in \C$ and $M\in \M$.
\begin{equation}\label{eq:unimodularModule}
\begin{gathered}
\xymatrix{
    X\tr M \ar[d]_{\fu_{X\tr M}} \ar[r]^-{\id_X\tr \fu_M} & X\tr \dS\;\dN(M) \\
    \dS\;\dN(X\tr M) \ar[r]_-{\dS(\fn_{X,M})} & \dS(X\rv\rv \tr \dN(M)) \ar[u]_{\fs_{\textnormal{$X\rv\rv, \dN(M)$}}} 
}
\end{gathered}
\end{equation}
An exact, left $\C$-module category is called \textit{unimodular} if it admits a unimodular structure $\fu$. 
\end{definition}

\begin{lemma}\label{lem:uniDefn}
Definitions~\ref{defn:unimodular} and \ref{defn:uniModCat} of a unimodular module category are equivalent.   
\end{lemma}
\begin{proof}
According to Definition~\ref{defn:unimodular}, $\M$ is unimodular if and only if $\Rex_{\C}(\M)$ is unimodular.
By equation (\ref{eq:uniModCat}), $\Rex_{\C}(\M)$ is unimodular if and only if there exists a $\C$-module natural isomorphism $\fu:\id_{\M}\rightarrow \dS\,\dN$, that is a unimodular structure on $\M$. Thus, the claim follows.
\end{proof}

\begin{remark}\label{rem:uniDefn}
Since $\dN$ is an equivalence with quasi-inverse $\overline{\dN}$ by Theorem~\ref{thm:Nak}(b), $\M$ is unimodular if and only if there exists a natural isomorphism $\overline{\dN}\cong \dS$ of $\C$-module functors. Thus, Definition~\ref{defn:uniModCat} is the same as the one suggested in \cite[Remark~4.27]{fuchs2020eilenberg}.
\end{remark}

Next, we present some examples of unimodular module categories.
\begin{example}\label{ex:uniModCat}
\upshape{(i)} Let $\C$ be a finite tensor category. Then, $\M=\C$ considered as a left $\C$-module category is unimodular if and only if $\C$ is a unimodular tensor category. 
For instance, if $\C$ is a semisimple finite tensor category, it is unimodular by \cite[Corollary~6.4]{etingof2004analogue}.
\smallskip

\noindent
\upshape{(ii)} Let $\C$ be a nondegenerate (\textit{i.e.} global dimension of $\C$ is nonzero) fusion category. For instance, the category $\Rep(H)$ for $H$ a semisimple and cosemisimple Hopf algebra is a nondegenerate fusion category. 
If $\M$ is a semisimple left $\C$-module category, then by \cite[Theorem~2.18]{etingof2005fusion}, $\Rex_{\C}(\M)$ is multifusion. Hence, by Definition~\ref{defn:unimodular}, $\M$ is a unimodular left $\C$-module category.

\smallskip

\noindent
\upshape{(iii)} Let $H$ be a finite dimensional Hopf algebra. Then the category $\C=\Rep(H)$ is a finite tensor category. The forgetful functor $F:\Rep(H)\rightarrow\Vect$ turns $\Vect$ into a left $\Rep(H)$-module category. By \cite[Example~7.12.26]{etingof2016tensor}, we have that $\Rex_{\C}(\Vect) \cong \Rep(H^*)$, where $H^*$ is the dual Hopf algebra of $H$. Therefore, $\Vect$ is a unimodular $\Rep(H)$-module category if and only if $H^*$ is a unimodular Hopf algebra. This happens if and only if the distinguished grouplike element $g_H$ of $H$ is equal to its unit $1_H$. 
\end{example}

In Section~\ref{sec:unimodular}, we generalize Example~\ref{ex:uniModCat}(iii) and classify unimodular $\Rep(H)$-module categories over for any finite dimensional Hopf algebra $H$. 

\smallskip

The next remark shows that tensor categories that are not unimodular can admit unimodular module categories.

\begin{remark}
Let $\C$ be a finite tensor category that is not unimodular ($D_{\C}\ncong \unit$). For instance, one can take $\C$ to be the category of representation of the Taft algebra. Now, consider the category $\D=\C\boxtimes \C^{\rev}$. Then, we get that
\begin{equation*}
    D_{\C\boxtimes \C^{\rev}} = \odN_{\C\boxtimes \C^{\rev}}(\unit_{\C\boxtimes \C^{\rev}}) 
    \stackrel{(\dagger)}{\cong}  
    (\odN_{\C}\boxtimes \odN_{\C^{\rev}})(\unit\boxtimes \unit) 
    \stackrel{(\ddagger)}{\cong} 
    \odN_{\C}(\unit)\boxtimes \odN_{\C}(\unit) 
    = 
    D_{\C}\boxtimes D_{\C} 
    \stackrel{}{\ncong}
    \unit\boxtimes\unit.
\end{equation*}
Here, the isomorphism $(\dagger)$ follows from \cite[Proposition~3.20]{fuchs2020eilenberg}. Since, $\C=\C^{\rev}$ as categories, we have that $\odN_{\C}=\odN_{\C^{\rev}}$. Thus, the isomorphism $(\ddagger)$ holds.
Consequently, we get that $\D$ is not unimodular. 
Now, consider the left $\D$-module category $\M:=\C$ with actions defined as 
\[(X\boxtimes Y) \tr M:= X\otimes M\otimes Y.\] 
Then, we claim that the $\D$-module category $\M$ is unimodular. Indeed, by \cite[7.13.8]{etingof2016tensor}, we have that $\Rex_{\D}(\M)\cong \Z(\C)$. Further, by \cite[Proposition~8.6.3]{etingof2016tensor}, $\Z(\C)$ is factorizable and by \cite[8.10.10]{etingof2016tensor}, factorizable finite tensor categories are unimodular. Thus, $\Z(\C)$ is unimodular. Equivalently, $\M$ is a unimodular $\D$-module category. 
\end{remark}


Next, we collect a few basic results about unimodular module categories. The following result shows that unimodular module categories are closed under direct sums. 

\begin{lemma}\label{lem:uniDirectSum}
    Suppose that $\M=\oplus_{i\in I} \M_i$ where the set $I$ is finite and each $\M_i$ is an indecomposable, exact left $\C$-module category. Then $\M$ is unimodular if and only each $\M_i$ is unimodular.
\end{lemma}
\begin{proof}
Since $\M$ is decomposable, $\D:=\Rex_{\C}(\M)$ is a multitensor category. Thus, we can write $\D$ as a direct sum of its component subcategories $\D_{ij}:=\Rex_{\C}(\M)_{ij}$.
But $\D_{ij}=\Rex_{\C}(\M_i,\M_j)$ by Lemma~\cite[Lemma~7.12.6]{etingof2016tensor}. Now, by the following equivalences, the claim follows.
\begin{align*}
    \M\; \text{unimodular} \;\; 
    \iff \; D_{\D}\cong_{\C}\id_{\M} \; \;
    & \stackrel{(\textnormal{\ref{lem:uniMulti}})}{\iff} \;\; D_{\D_{ii}}\cong_{\C} \id_{\M_i} \; \forall i\in I \\
    &\iff \;\; \M_i \;\text{unimodular}\; \forall i\in I.
    \qedhere
\end{align*}
\end{proof}

\begin{lemma}
Let $\M$ be an indecomposable, exact, left $\C$-module category. Then, a unimodular structure on $\M$, if it exists, is unique up to a scalar multiple.
\end{lemma}
\begin{proof}
Since $\M$ is indecomposable, $\id_{\M}$ is a simple object in $\Rex_{\C}(\M)$. Now, given two unimodular structures $\fu$ and $\fu'$ on $\M$, $\fu'\circ \fu^{-1}$ is an endomorphism of $\id_{\M}$. Hence, by Schur's Lemma,
\begin{equation*}
    \fu'\circ \fu^{-1} = k\;  \id_{\M} \text{ for some $k\in \kk$} \hspace{0.5cm} \Rightarrow \hspace{0.5cm} \fu'=k\; \fu . \qedhere
\end{equation*} 
\end{proof}

\begin{proposition}
Let $\M$ be a unimodular left $\C$-module category satisfying $\dN\cong\id_{\M}$. Then, for any $\M\in\M$, the internal End object $\uHom(M,M)$, is a Frobenius algebra in $\C$.
\end{proposition}
\begin{proof}
Since $\M$ is unimodular, we have a natural isomorphism $\fu:\id_{\M}\rightarrow \dS\; \dN$. We also have a natural isomorphism $\tau : \dN \rightarrow \id_{\M}$. By combining these two isomorphisms, we get a natural isomorphism $p:\id_{\M}\xrightarrow{\fu} \dS\;\dN \xrightarrow{ \dS\; \tau} \dS$, thereby providing an isomorphism $p_M:M\rightarrow\dS(M)$ for all $M\in \M$. Now, by \cite[Theorem~3.14]{shimizu2019relative}, it follows that $\uHom(M,M)$ is a Frobenius algebra in~$\C$. 
\end{proof}

\begin{remark}
A finite linear category $\M$ is said to be \textit{symmetric Frobenius} if $\M$ is equivalent to the category of modules over a symmetric Frobenius algebra $A$ \cite[Definition~3.23]{fuchs2020eilenberg}. This definition is justified because the property of being a symmetric Frobenius algebra is Morita invariant. By \cite[Proposition~3.24]{fuchs2020eilenberg}, $\M$ is symmetric Frobenius if and only if $\dN_{\M}\cong \id_{\M}$. Thus, symmetric Frobenius module categories provide natural candidates to which the above proposition can be applied.    
\end{remark}

\subsection{Frobenius algebras from unimodular module categories}\label{subsec:frobUni}
In this section, we use unimodular module categories to provide a construction of (separable, special) Frobenius algebras in the Drinfeld center $\Z(\C)$. We accomplish this by constructing appropriate Frobenius monoidal functors with target $\Z(\C)$. To construct such functors, we employ the strategy outlined in \cite{yadav2022frobenius}. Namely, we consider the right adjoint of the functor $\Psi$ defined below.

\begin{definition}\cite{shimizu2020further}\label{defn:Psi}
Let $\C$ be a finite tensor category and $\M$ a left $\C$-module category. Then, we define a functor $\Psi$ as 
\begin{equation*}\label{eq:4Psi}
\Psi:=\Psi_{\M}:\Z(\C) \rightarrow \Rex_{\C}(\M), \hspace{1cm} (X,\sigma) \mapsto  (X\tr -,s^{\sigma}),
\end{equation*} 
where the left $\C$-module structure of the functor  $X\tr -$ is given by
\[s^{\sigma}_{Y,M}: Y\tr (X\tr M) = (Y\otimes X)\tr M \xrightarrow{\sigma_Y\tr \id_M} (X\otimes Y)\tr M = X\tr (Y\tr M).\]
\end{definition}

Next, we recall from \cite[Lemma~4.7]{yadav2022frobenius} that we can write $\Psi = U'\circ \Omega \circ \Theta$ where,
\begin{itemize}
    \item $\Theta: \Z(\C)\rightarrow \Z(\Rex_{\C}(\M)^{\rev})$ is Schauenburg's equivalence of braided categories \cite[Theorem~3.3]{schauenburg2001monoidal},
    \smallskip
    \item $\Omega:\Z(\Rex_{\C}(\M)^{\rev}) \rightarrow \Z(\Rex_{\C}(\M))^{\mir}$ is the equivalence of braided categories provided by \cite[Exercise~8.5.2]{etingof2016tensor}, and
    \smallskip
    \item $U':\Z(\Rex_{\C}(\M))^{\mir} \rightarrow \Rex_{\C} (\M)$ is the functor that forgets the half-braiding. Let $R$ denote the right adjoint of $U'$. 
\end{itemize}

\smallskip

In the following discussion, the algebra object $\Psi^{\ra}(\id_{\M})\in\Z(\C)$ will be very important. By Theorem~\ref{thm:adjAlg}(a), $\Psi^{\ra}(\id_{\M})$ is a commutative algebra in $\Z(\C)$. To start, consider the following result, which follows from the work in \cite{shimizu2020further}.

\begin{theorem}\label{thm:PsiRaCommFrob}
Let $\C$ be a finite tensor category and $\M$ a finite left $\C$-module category. Then, $\M$ is unimodular if and only if $\Psi^{\ra}(\id_\M)$ is a Frobenius algebra in $\Z(\C)$. 
\end{theorem}
\begin{proof} 
Let $\Theta^{-1}, \Omega^{-1}$, respectively, denote the quasi-inverse of the equivalences $\Theta, \Omega$. 
Then, 
\begin{equation}\label{eq:PsiRa}
    \Psi^{\ra}\cong  \Theta^{\ra} \circ \Omega^{\ra} \circ U'^{\ra} \cong \Theta^{-1} \circ \Omega^{-1} \circ R
\end{equation}
Suppose that $\M$ is indecomposable. Then, we get the following equivalences.
\begin{align*}
    \Psi^{\ra}(\id_{\M})\in\Frob(\Z(\C)) 
    \stackrel{(\diamondsuit)}{\iff} 
    R(\id_{\M})\in \Frob(\Z(\Rex_{\C}(\M))) 
    & \stackrel{(\spadesuit)}{\iff} 
    \Rex_{\C}(\M)\; \textnormal{unimodular} \\[-0.13cm]
    & \stackrel{(\textnormal{\ref{lem:uniDefn}})}{\iff} 
    \M \;\text{unimodular}.  
\end{align*}
The equivalence $(\diamondsuit)$ holds because $\Theta^{-1}$ and $\Omega^{-1}$ are monoidal equivalences, and thereby they preserve Frobenius algebras. The equivalence $(\spadesuit)$ follows from Theorem~\ref{thm:adjAlg}(b) because $\Rex_{\C}(\M)$ is a finite tensor category.

Now suppose that $\M$ is decomposable and $\M=\oplus_{i\in I}\M_i$ where $\M_i$ are indecomposable $\C$-module categories and the set $I$ is finite. First observe that 
\begin{equation*}
    \Psi^{\ra}_{\M}(\id_{\M}) = \Psi^{\ra}_{\M}(\oplus_{i\in I} \id_{\M_i}) = \oplus_{i\in I}\Psi^{\ra}_{\M_i}(\id_{\M_i}) 
\end{equation*}
Now by the following equivalences, the claim follows. Below, we write $\Frob$ to denote $\Frob(\Z(\C))$.
\begin{align*}
    \Psi^{\ra}_{\M}(\id_{\M})\in\Frob  
    \stackrel{(\textnormal{\ref{prop:frobDirectSum}})}{\iff}  
    \Psi^{\ra}_{\M_i}(\id_{\M_i}) \in \Frob\; \forall i\in I  
    & \iff 
    \M_i \;\text{unimodular}\; \forall  i\in I  \\[-0.16cm]
    &  \stackrel{(\textnormal{\ref{lem:uniDirectSum}})}{\iff} 
    \M \; \text{unimodular} .  \qedhere
\end{align*}

\end{proof}

In the following theorem, we collect many  characterizations of unimodular module categories.
This result also highlights the importance of the functor $\Psi^{\ra}$ to the problem of constructing commutative Frobenius algebras in the Drinfeld center. 

\begin{theorem}\label{thm:equivalentUni}
Let $\C$ be a finite tensor category and $\M$ an exact, left $\C$-module category. Then, the following are equivalent.
\begin{enumerate}
    \item[\upshape{(a)}] $\M$ is a unimodular module category.

    \item[\upshape{(b)}] $\Rex_{\C}(\M)$ is a unimodular multitensor category.

    \item[\upshape{(c)}] $\dS\;\dN\cong \id_{\M}$ as left $\C$-module functors.

    \item[\upshape{(d)}] $\Psi^{\ra}(\id_{\M})$ is a Frobenius algebra in $\Z(\C)$.
\end{enumerate}
If furthermore, $\M$ is indecomposable, then above conditions are equivalent to the following.
\begin{enumerate}
    \item[\upshape{(e)}] $\Psi^{\ra}$ is a Frobenius monoidal functor.
\end{enumerate}
\end{theorem}
\begin{proof}
(a)$\Leftrightarrow$(c) and (a) $\Leftrightarrow$(b) are clear from Definition~\ref{defn:uniModCat} and Lemma~\ref{lem:uniDefn}, respectively. Also, (a)$\Leftrightarrow$(d) is known by Theorem~\ref{thm:PsiRaCommFrob}. When $\M$ is indecomposable, (d)$\Leftrightarrow$(e) follows from \cite[Theorem~1.2(i)]{yadav2022frobenius}. 
\end{proof}

Consequently, by applying \cite[Theorem~1.2(i)]{yadav2022frobenius} to the functor $\Psi$, we get the following result.

\begin{corollary}\label{cor:PsiRaSS}
Let $\C$ be a finite tensor category and $\M$ be an indecomposable, unimodular left $\C$-module category. Then, $\Psi^{\ra}$ is a separable (resp. special) Frobenius monoidal functor if and only if the Frobenius algebra $\Psi^{\ra}(\id_{\M})$ in $\Z(\C)$ is separable (resp. special).  \qed
\end{corollary}

For future use, we also record the following result.

\begin{lemma}\label{lem:connected}
Let $\M$ be an indecomposable left $\C$-module category. Then, $\Psi^{\ra}(\id_{\M})$ is connected. 
\end{lemma}
\begin{proof}
The proof follows from the following computation. 
\begin{equation*}
    \Hom_{\Z(\C)}(\unit_{\Z(\C)},\Psi^{\ra}(\id_{\M}))\cong \Hom_{\Rex_{\C}(\M)}(\Psi(\unit_{\Z(\C)}),\id_{\M}) = \Hom_{\Rex_{\C}(\M)}(\id_{\M},\id_{\M}) \cong \kk. \qedhere
\end{equation*} 
\end{proof}


\subsection{Pivotal case}\label{subsec:pivUni}
Next, we consider the pivotal case when $\C$ is a pivotal finite tensor category. This assumption is needed in order to construct symmetric Frobenius algebras in $\Z(\C)$.

Let $\C$ be a pivotal tensor category with pivotal structure $\mathfrak{p}:\id_{\C}\xrightarrow{\cong} \lv\lv(-)$. Recall that, by Theorem~\ref{thm:Serre}(a), if $\M$ is an exact left $\C$-module category, then the (right) relative Serre functor $\dS$ of $\M$ exists. Then, $\dS$ is a left $\C$-module functor with module constraint given by 
\[\dS(X\tr M) \xrightarrow{\fs_{X,M}} \lv \lv X\tr \dS(M) \xrightarrow{\mathfrak{p}_X^{-1} \tr \id_{M} }  X\tr \dS(M) .\]  
This allows one to define a pivotal structure on an exact $\C$-module category.

\begin{definition}\cite[Definition~3.11]{shimizu2019relative}\label{defn:pivotalModule}
Let $\C$ be a pivotal tensor category. A \textit{pivotal structure} on an exact left $\C$-module category $\M$ is a left $\C$-module natural isomorphism $\widetilde{\fp} :\id_{\M}\rightarrow \dS$. A \textit{pivotal left $\C$-module category} is an exact left $\C$-module category equipped with a pivotal structure.
\end{definition}

Shimizu proved the following interesting property of pivotal module categories.

\begin{theorem}\cite[Theorem~3.13]{shimizu2019relative}\label{thm:ShiRexPivotal}
If $\C$ is a pivotal finite tensor category and $\M$ is a pivotal left $\C$-module category, then $(\Rex_{\C}(\M))^{\rev}$ is a pivotal finite multitensor category.    \qed
\end{theorem}

We use the notation $F^{\lla}:=(F^{\la})^{\la}$ and $F^{\rra}=(F^{\ra})^{\ra}$. Under the assumptions of Theorem~\ref{thm:ShiRexPivotal}, $\Rex_{\C}(\M)$ is also a pivotal monoidal category. In particular, its pivotal structure is given by:
\[ \fp^{\Rex_{\C}(\M)}_F:=(F \xrightarrow{F\circ \widetilde{\fp}}  F \circ \dS  \xrightarrow{\omega_{F^{\lla}}^{-1}} {\dS \circ F^{\lla}} \xrightarrow{\widetilde{\fp}^{-1} \circ F^{\lla}}  F^{\lla}), \] 
where $\omega_F:\dS_{\N}\circ F\rightarrow F^{\rra}\circ \dS_{\M}$ for $F\in \Rex_{\C}(\M)$ is the natural isomorphism of functors from \cite[Theorem~3.10]{shimizu2019relative}.
Next, we prove the two main result of this section.

\begin{theorem}\label{thm:PsiPivotal}
Let $\C$ be a pivotal finite tensor category and $\M$ be a indecomposable, pivotal, unimodular left $\C$-module category. Then $\Psi^{\ra}$ is a pivotal Frobenius monoidal functor.
\end{theorem}
\begin{proof}
By \cite[Theorem~1.2(ii)]{yadav2022frobenius}, $\Psi^{\ra}$ is a pivotal functor if $\Psi^{\ra}(\id_{\M})$ is a symmetric Frobenius algebra in $\Z(\C)$. 
First, recall that $\Psi= U'\circ \Omega \circ \Theta$. Thus, by (\ref{eq:PsiRa}),
\[ \Psi^{\ra}(\id_{\M}) \cong \Theta^{-1}\circ \Omega^{-1}\circ (U')^{\ra}(\id_{\M}). \]
As $\Rex_{\C}(\M)$ is pivotal finite tensor category, by Theorem~\ref{thm:adjAlg}(c), $(U')^{\ra}(\id_{\M})$ is a symmetric Frobenius algebra.  
It is straightforward that $\Omega$ is a pivotal equivalence and by \cite[Theorem~5.16]{spherical2022}, $\Theta$ is a pivotal equivalence as well. Hence, $\Omega^{-1}, \;\Theta^{-1}$ are also pivotal functors, and, by Lemma~\ref{lem:pivShimizuB}, so is $\Omega^{-1}\circ \Theta^{-1}$. As pivotal functors preserve symmetric Frobenius algebras, we conclude that  $\Psi^{\ra}(\id_{\M})$ is a symmetric Frobenius algebra in $\Z(\C)$. So, we are done by \cite[Theorem~1.2(ii)]{yadav2022frobenius}.
\end{proof}

\begin{theorem}\label{thm:PsiSpecial}
Let $\C$ be a pivotal finite tensor category and $\M$ be an indecomposable, pivotal, unimodular left $\C$-module category. Then, $\Psi^{\ra}$ is a special, pivotal Frobenius monoidal functor if and only if $\dim(\Psi^{\ra}(\id_{\M}))\neq 0$.
\end{theorem}
\begin{proof}
By Theorem~\ref{thm:PsiPivotal}, we know that $\Psi^{\ra}$ is a pivotal Frobenius pivotal functor. As such functors preserve symmetric Frobenius algebras, and $\id_{\M}$ is a symmetric Frobenius algebra in $\Rex_{\C}(\M)$, it follows that $\Psi^{\ra}(\id_{\M})$ is a symmetric Frobenius algebra in $\Z(\C)$.

$(\Rightarrow)$ 
Suppose that $\Psi^{\ra}$ is special with nonzero constants $\beta_0,\beta_2$. Then, by Lemma~\ref{lem:dimPiv},
\begin{equation*}
    \dim(\Psi^{\ra}(\id_{\M})) = \beta_0\beta_2 \dim(\id_{\M}) = \beta_0\beta_2 \neq 0.
\end{equation*}

$(\Leftarrow)$
We know that $\Psi^{\ra}(\id_{\M})$ is a connected (Lemma~\ref{lem:connected}) Frobenius algebra of nonzero dimension. Thus, by Lemma~\ref{lem:speFrob}, $\Psi^{\ra}(\id_{\M})$ is special Frobenius. Hence, by Corollary~\ref{cor:PsiRaSS}, $\Psi^{\ra}$ is a special, pivotal Frobenius monoidal functor.
\end{proof}


\section{Unimodular exact comodule algebras}\label{sec:comoduleAlg}
In this section, we classify unimodular module categories over the finite tensor category $\C=\Rep(H)$ for $H$ a finite dimensional Hopf algebra. By \cite{andruskiewitsch2007module}, every left $\C$-module category is of the form $\M=\Rep(A)$ for $A$ a left $H$-comodule algebra. In this setting, $A$ is called \textit{exact} (resp., \textit{indecomposable}) if the $\C$-module category $\M$ is exact (resp., indecomposable). We call an indecomposable, exact $H$-comodule algebra $A$ \textit{unimodular} if $\Rep(A)$ is a unimodular $\Rep(H)$-module category. The main result of this section, Theorem~\ref{thm:uniclassification}, provides an explicit classification of unimodular $H$-comodule algebras. To accomplish this, we describe in detail the functor $\dS_{\M} \dN_{\M}$. However, by Theorem~\ref{thm:NakSerre}, $\dS_{\M}\cong_{\C} D_{\C}\tr \dN_{\M}$. Thus, we need to describe $D_{\C}$ and $\dN_{\M}$.

\smallskip

We start by providing background material on exact comodule algebras $A$ and explicitly describing the twisted $\C$-module structure of the Nakayama functor, $\dN_{\Rep(A)}$, of $\Rep(A)$ in Section~\ref{subsec:hopfNakayama}. Using this, we obtain an explicit description of the distinguished invertible object $D_{\Rep(H)}$, the Radford isomorphism and the Serre functor $\dS_{\Rep(A)}$ in Section~\ref{subsec:hopfRadford}. In Section~\ref{subsec:hopfUnimodular}, we describe the functor $\dS_{\Rep(A)} \dN_{\Rep(A)}$ and its left $\C$-module structure. Using this, we provide a criteria for unimodularity of $\M$, or equivalently the unimodularity of $A$, in terms of certain invertible elements in $A$ which we call as \textit{unimodular elements}, see Definition~\ref{defn:uniElement}. Finally, in Section~\ref{subsec:Taft}, we discuss in detail the example of Taft algebras.

\smallskip

\begin{notation}
Throughout this section, we work with vector spaces over a field $\kk$. We use the notation $X^*$ to denote the vector space $\Hom_{\kk}(X,\kk)$. For any finite dimensional vector space $X$, we let $\phi_X:X\rightarrow X^{**}$ denote the canonical isomorphism map. Let $\Delta,\varepsilon, S$ denote the comultiplication, counit and antipode of $H$ respectively. We will frequently use the Sweedler notation for calculations. For $X\in \Rep(H)$, we have that $\lv X=X\rv=X^*$ as vector space with $H$-actions as described below,
\begin{equation*}
    (h\cdot f)(x):= f(S(h)\cdot x),\;\;\;\;\;\;\;\ (h\cdot f')(x):= f'(S^{-1}(h)\cdot x)
\end{equation*}
for all $h\in H,f \in\lv X, f'\in X\rv$ and $x\in X$. 
Similarly, for $\phi_X(x)\in \lv\lv X$, $h\cdot \phi_X(x):=\phi_X(S^{2}(h)\cdot x)$ and for $\phi_X(x)\in X\rv\rv$, $h\cdot \phi_X(x):=\phi_X(S^{-2}(h)\cdot x)$.
\end{notation}


\subsection{Exact comodule algebras and the Nakayama functor}\label{subsec:hopfNakayama}
A \textit{left $H$-comodule algebra} is a left $H$-comodule $(A,\rho:A\rightarrow H\otimes A)$ with an algebra structure such that the multiplication and unit maps are $H$-comodule maps, that is,
\begin{equation*}
    \rho(a a') = a_{(-1)}a'_{(-1)} \ok a_{(0)} a'_{(0)}, \;\;\;\;\;\; \rho(1_A) = 1_H\ok 1_A \;\;\;\;\;\;\; (\forall \; a,a'\in A)
\end{equation*}
where $\rho(a)$ is denoted as $a_{(-1)}\otimes a_{(0)}\in H\otimes A$.
Then $\Rep(A)$ is a left $\Rep(H)$-module category via the action given by $\tr :\Rep(H) \times \Rep(A) \rightarrow \Rep(A)$ where $X\tr M=X\ok M$ as vector space, and the $A$-action on $X\tr M$ is defined as 
\[ a\cdot(x\otimes m)=a_{(-1)}\cdot x\otimes a_{(0)}\cdot m \hspace{1cm} (a\in A,\; x\in X,\; m\in M).  \]
Above, the first $\cdot$ is the action of $H$ on $X$ and the second $\cdot$ is the action of $A$ on $M$. 

\smallskip

By \cite[Lemma~4.5]{shimizu2019relative} (see also \cite{skryabin2007projectivity}), every exact left $H$-comodule algebra $A$ is a Frobenius algebra. 
Thus, we can endow $A$ with a \textit{Frobenius system}, that is, a triple $(\lambda_A,\{a^i \}, \{ b_i\})$ with $1\leq i\leq r=\text{dim}(A)$. Here $\lambda_A:A\rightarrow \kk$ is a linear map and $\{a^i\},$ $\{ b_i\}$ are two bases of $A$ such that $\langle \lambda_A, a^i b_j\rangle = \delta_{i,j}$ for all $i,j=1,\ldots, r$. 
The \textit{Nakayama automorphism} of $A$ (with respect to $\lambda_A$) is the unique algebra automorphism $\nu:= \nu_A:A\rightarrow A$ characterized by 
\begin{equation}\label{eq:NakayamaAuto}
    \langle \lambda_A,ab\rangle = \langle \lambda_A,\nu_A(b)a\rangle \hspace{1cm} (a,b\in A).
\end{equation}
By \cite[Lemma~4.2]{shimizu2019relative} , the following equalities hold for all $c\in A$:
\begin{equation}\label{eq:FrobSystem}
\langle\lambda_A,a^i\rangle b_i = 1_A = \langle\lambda_A,b_i\rangle a^i, \hspace{0.5cm} a^i c \otimes b_i = a^i\otimes cb_i, \hspace{0.5cm} \nu_A(c)  a^i\otimes b_i = a^i\otimes b_i c.   
\end{equation}
The formulae above omit the summations over $i$, and we will continue to do so.

\begin{notation}
We will use the following notations going forward.
\begin{enumerate}
    \item For a vector space $M$, we will denote the basis of $M$ by $m_i$ and the dual basis of $M^*$ by $m^i$. These satisfy $\langle m^i, m \rangle m_i = m$ for all $m\in M$.
    
    \item If $(X,\cdot)$ is a left $H$-module and $f:H'\rightarrow H$ is an algebra map. We use the notation $ ({}_{f}X, \cdot_{f})$ to denote the $H'$-module $X$ with action given by $h'\cdot_{f} x:= f(h')\cdot x$ for $h'\in H',x\in X$.

    \item The map $\phi_M^N:M^*\oA N \xrightarrow{\sim} \Hom_A(M,N)$ is an isomorphism given by $m^*\oA n\mapsto \langle m^*,?\rangle n$ with inverse $f\mapsto m^i\oA \langle f,m_i\rangle$.
\end{enumerate}
\end{notation}

\subsubsection{Nakayama functor of $\Rep(A)$}\label{subsubsec:hopfNakayama}
Let $A$ be a Frobenius algebra. We first provide a description of the Nakayama functor of $\Rep(A)$.

\begin{theorem}\label{thm:RepANak}
    Let $A$ be a Frobenius algebra with Nakayama automorphism $\nu$. Then, we have that 
    \begin{equation}\label{eq:nu}
        \dN_{\Rep(A)}(M)= \int^{N\in \Rep(A)} \HomA(M,N)^* \; \btr N = {}_{\nu}M \;\;\;\;\;\;\; (M\in\Rep(A)). 
    \end{equation}
    The projection maps  $\bi_{M,N}:\HomA(M,N)^* \; \btr N \rightarrow {}_{\nu}M$ of the coend are given by 
    \begin{equation}\label{eq:bi}
        \bi_{M,N}(\xi\ok n) = \langle \xi, \phi_M^N(m^i \oA a^j\cdot n)  \rangle \nu(b_j )\cdot m_i
    \end{equation}
    for all $\xi\in\HomA(M,N)^*$ and $n\in N$.
\end{theorem}

While this result is well-known to the experts, we could not find a direct proof of it in the literature. So, for the reader's convenience, we provide a proof in Appendix~\ref{app:ARepANak}.

\smallskip

Now suppose that $A$ is an exact left $H$-comodule algebra. Then, by Theorem~\ref{thm:Nak}(a), $\dN_{\Rep(A)}$ is a twisted left $\Rep(H)$-module functor. The following result explicitly describes this structure.

\begin{theorem}\label{thm:RepAtwistedLeft}
Let $A$ be an exact left $H$-comodule algebra.
The twisted left $\Rep(H)$-module structure $\fnl_{X,M}:{}_{\nu}(X\tr M) \rightarrow X\rv\rv \tr {}_{\nu}M $ of the Nakayama functor $\dN$ of $\Rep(A)$ is given by
\begin{equation}\label{eq:fnl}
    \fnl_{X,M}(x \ok m)= \langle \lambda_A, a^i_{(0)}\rangle \phi_X(S^{-1}(a^i_{-1})\cdot x) \ok \nu(b_i)\cdot m.  
\end{equation} 
The inverse of $\fnl_{X,M}$ is given by 
\begin{equation}\label{eq:ofnl}
    \ofn^l_{X,M}(\phi_X(x) \ok m) = \langle \lambda_A,a^i_{(0)} \rangle \nu(b_i)\cdot\big( S^{-2}(a^i_{(-1)}) \cdot x \ok m \big).
\end{equation}
\end{theorem}

Now, let $B$ be a right $H$-comodule algebra. Then $\Rep(B)$ is a right $\Rep(H)$-module category. In this case, the Nakayama functor $\dN_{\Rep(B)}$ is a twisted right $\Rep(H)$-module functor and the following result describes this structure.

\begin{theorem}\label{thm:RepAtwistedRight}
Let $B$ be an exact right $H$-comodule algebra.
The twisted right $\C$-module structure $\fnr_{X,M}:  {}_{\nu}(M\tl X) \rightarrow  {}_{\nu}M \tl \lv\lv X$ of the Nakayama functor $\dN$ of $\Rep(B)$ is given by
\begin{equation}\label{eq:fnr}
    \fnr_{X,M}(m \ok x)= \langle \lambda_B, a^i_{(0)}\rangle \nu(b_i)\cdot m \ok \phi_X(S(a^i_{(1)}) \cdot x).
\end{equation}
The inverse of $\fnr_{X,M}$ is given by 
\begin{equation}\label{eq:ofnr}
    \ofn^r_{X,M}(m\ok \phi_X(x)) = \langle\lambda_B,a^i_{(0)} \rangle \nu(b_i) \cdot \big( m \ok S^2(a^i_{(1)}) \cdot x \big).
\end{equation}
\end{theorem}
\noindent
Building upon \cite[Theorem~7.3]{shibata2021modified}, a proof of Theorem~\ref{thm:RepAtwistedRight} is provided in Appendix~\ref{app:ARepAtwisted}. Then, Theorem~~\ref{thm:RepAtwistedLeft} is proved by applying Theorem~\ref{thm:RepAtwistedRight} to the exact right $H$-comodule algebra $A^{\op}$.


\subsection{Radford isomorphism} \label{subsec:hopfRadford}
In this section, we calculate the distinguished invertible objects (in Section~\ref{subsubsec:distinguished}) and the Radford isomorphism (in Section \ref{subsubsec:Radford}) for the finite tensor category $\Rep(H)$. This will allow us to explicitly describe the Serre functor of the module category $\Rep(A)$ (in Section~\ref{subsubsec:hopfSerre}).
We note the Radford isomorphism for $\Rep(H)$ and a Serre functor of $\Rep(A)$ are known by prior work of Shimizu \cite{shimizu2019relative}. Our purpose of calculating them again below is to ensure that our conventions are consistent throughout for future use here and in subsequent works.

\subsubsection{(Co)integrals} \label{subsubsec:cointegrals}
To start, we set notations for (co)integrals from \cite[Chapter~10]{radford2011hopf}. A \textit{left integral} is a nonzero element $\Lambda\in H$ satisfying $h \Lambda = \Lambda \varepsilon(h)$ for all $h\in H$. Then there exists a unique algebra map $\alpha_H:H\rightarrow \kkx$, called the \textit{distinguished character}, such that $\Lambda h= \langle \alpha_H,h\rangle \Lambda$ holds for all $h\in H$. Thus, we have that
\begin{equation}\label{eq:alphaGrouplike}
    \langle \alpha_H, h h'\rangle = \langle\alpha_H, h\rangle \langle\alpha_H,h' \rangle, \;\;\;\;\;\;\; \langle \alpha_H, 1_H\rangle =1 \;\;\;\;\;\;\;\;\; \text{for all} \; h,h'\in H.
\end{equation}
A Hopf algebra $H$ is called \textit{unimodular} if $\alpha_H=\varepsilon$.  
A \textit{right cointegral} of $H$ is a nonzero element $\lambda_H\in H^*$ satisfying
\begin{equation}\label{eq:rightCointegral1}
    \langle\lambda_H,h_{(1)}\rangle h_{(2)} = \langle \lambda_H,h\rangle 1_H \; \text{for all}\; h\in H.
\end{equation} 
Then there is a unique grouplike element $g_H\in H$ satisfying 
\begin{equation}\label{eq:rightCointegral2}
    h_{(1)} \langle\lambda_H,h_{(2)} \rangle = \langle \lambda_H,h\rangle g_H \; \text{for all} \; h\in H.
\end{equation} 
From here on, we fix a left integral $\Lambda$ and right cointegral $\lambda_H$ satisfying $\langle \lambda_H,\Lambda \rangle = 1$. Also, set $\overline{\alpha}_H :=\alpha_H\circ S$ and $\ogH := g_H^{-1}$.


\subsubsection{Distinguished invertible object of $\Rep(H)$}\label{subsubsec:distinguished}
It is well known that finite dimensional Hopf algebras are Frobenius with the Frobenius form given by any right cointegral $\lambda_H$. By \cite[Theorem~3(a,b)]{radford1994trace}, we get that the Nakayama automorphism of $H$ is given by 
\begin{equation*}\label{eq:nuH}
    \nu_H(h) = \langle\alpha_H ,h_1\rangle S^2(h_2)
\end{equation*} 
and its inverse $\onu$ is given by
\begin{equation*}
    \onu_H(h) = S^2(\ogH h_1 \gH) \langle \alpha_H, \gH S(h_2) \ogH\rangle \stackrel{(\textnormal{\ref{eq:alphaGrouplike}})}{=} S^2(\ogH h_1 \gH) \langle \alpha_H, S(h_2) \rangle .
\end{equation*}

\begin{lemma}\label{lem:D}
    We have the following results for $\unit_{\Rep(H)}=\kk$.
\begin{enumerate}
    \item[\upshape{(a)}] $\dN_{\Rep(H)}(\kk)=\kk$ as a vector space with $H$-action given by $h\star c:= \langle \alpha_H,h\rangle c$. 
    \item[\upshape{(b)}] $D_{\Rep(H)}=\overline{\dN}_{\Rep(H)}(\kk)=\kk$ as vector space with $H$-action given by $h \ostar c:= \langle \alpha_H, S(h) \rangle c$.
\end{enumerate}
\end{lemma}
\begin{proof}
By Theorem~\ref{thm:RepANak}, $\dN(\kk)={}_{\nu}\kk$. Since, the $H$-action on $\kk$ is given by $h\cdot c=\langle \varepsilon_H,h\rangle c$, we get that that the $H$-action $\star$ on ${}_{\nu_H}\kk$ is given by 
\begin{equation*}
    h\star c 
    =\langle\varepsilon_H,\nu_H(h)\rangle c 
    = \langle\alpha_H,h_1\rangle \langle\varepsilon_H,S^2(h_2) \rangle c 
    = \langle \alpha_H,h_1\rangle \langle  \varepsilon_H , h_2\rangle 
    = \langle\alpha_H, h\rangle c.
\end{equation*}
As the functor $\odN$ is the quasi-inverse of $\dN$, it clear that $\odN(M)={}_{\onu}M$. Now using the same argument as above, the claim follows.
\end{proof}

\subsubsection{Radford isomorphism of $\Rep(H)$}\label{subsubsec:Radford}
The following material is based on the discussion in \cite[\S6.3]{shibata2021modified}. The Hopf algebra $H$ is a left and right $H$-comodule algebra. Thus, the category $\Rep(H)$ is $ \Rep(H)$-bimodule category. By using Theorems~\ref{thm:RepAtwistedLeft} and \ref{thm:RepAtwistedRight} with $A=B=H$, one obtains the following map.
\begin{equation}\label{eq:gX}
    g_X:\dN(\unit)\tl \lv\lv X \xrightarrow{\ofn^r_{X,\unit }} \dN(\unit \tl X) \xrightarrow{\textnormal{flip}} \dN(X\tr \unit) \xrightarrow{\fn^l_{X,\unit}} X\rv\rv \tr \dN(\unit) 
\end{equation}
As explained in \cite[Remark~4.11]{shimizu2017ribbon}, the Radford isomorphism $\fR_X$ can be described using the map $g_X$ as follows. 
First recall that by \cite[Lemma~4.11]{fuchs2020eilenberg}, $D =\overline{\dN}(\unit)$ is the right dual of $\dN(\unit)$. The evaluation and coevaluation maps are trivial identity maps. Then, $\fR_X$ is equal to the following composition.
\begin{equation}\label{eq:radIso1}
    \lv\lv X\otimes D 
    \xrightarrow{\coev_D\otimes \id}   
    D\otimes \dN(\unit) \otimes \lv\lv X \otimes D  
    \xrightarrow{\id\otimes g_X\otimes \id} 
    D\otimes X\rv\rv \otimes  \dN(\unit) \otimes D \xrightarrow{\id\otimes \ev_D}
    D\otimes X\rv\rv
\end{equation}
Thus, we first calculate the map $g_X$ below.

\begin{lemma}
The map $g_X$ in $(\ref{eq:gX})$ is given by $c\ok \phi_X(x) \mapsto \phi_X(\ogH\cdot x) \ok c$ for all $x\in X$ and $ X\in \Rep(A)$.
\end{lemma}
\begin{proof}
Fix any Frobenius system $(\lambda_H, \{a^i\}, \{b_i \})$ of $H$ with $\lambda_H$ is a right integral of $H$. 
\begin{align*}
c\ok \phi_X(x) & \;\;\;\stackrel{(\textnormal{\ref{eq:gX}})}{\mapsto} \;\; \fn^l_{X,\unit}\circ \textnormal{flip}\circ \ofn^r_{X,\unit}(c\ok \phi_X(x)) 
\\[-0.3em]
& \;\;\;\stackrel{(\textnormal{\ref{eq:ofnr}})}{=}  \;\;\fn^l_{X,\unit} \circ \textnormal{flip} \big[\langle\lambda_H,a^i_{(1)}  \rangle \nu(b_i)\cdot \big(c \ok S^{2}(a^i_{1})\cdot x\big) \big] 
\\[-0.1em]
& \;\;\;\;\,\,=  \;\;\;\; \fn^l_{X,\unit}\big[\langle\lambda_H,a^i_{(1)}  \rangle \big(  \nu(b_i)_{(2)} S^{2}(a^i_{2})\cdot x  \ok \nu(b_i)_{(1)}\cdot c \big) \big] 
\\[-0.3em]
& \;\;\;\;\stackrel{(\textnormal{\ref{eq:fnl}})}{=}  \;\;\langle \lambda_H,a^j_{(2)}\rangle  \langle\lambda_H,a^i_{(1)}  \rangle  \phi_X(S^{-1}(a^j_{1}) \nu(b_i)_{(2)} S^{2}(a^i_{2}) \cdot x) \ok \nu(b_j) \nu(b_i)_{(1)}\cdot c
\\[-0.3em]
& \,\stackrel{(\textnormal{\ref{eq:rightCointegral1}},\textnormal{\ref{eq:rightCointegral2}})}{=} 
\langle \lambda_H,a^j\rangle  \langle\lambda_H,a^i\rangle 
\phi_X(S^{-1}(g_H) \nu(b_i)_{(2)} S^{2}(1_H) \cdot x) \ok \nu(b_j) \nu(b_i)_{(1)}\cdot c 
\\[-0.35em]
& \;\;\;\;\stackrel{(\textnormal{\ref{eq:FrobSystem}})}{=} \;\;\;  \phi_X(\ogH \cdot x) \ok c.   \qedhere
\end{align*}
\end{proof}

\begin{remark}
For every finite dimensional Hopf algebra $H$, Radford \cite{radford1976order} proved the following formula between the fourth power of antipode.
\begin{equation}\label{eq:S4}
    S^4(h) = g_H \; [\alpha_H(h_1)\, h_2 \, \overline{\alpha}_H(h_3)]\, \ogH.
\end{equation}
The fact that $g_X$ is a morphism of left $H$-modules is equivalent to above mentioned formula for the fourth power of antipode after plugging $h=\alpha_H\rightharpoonup S^{-2}(h)$ in (\ref{eq:S4}).
\end{remark}

We are now ready to provide a formula for the Radford isomorphism.
\begin{proposition}
The Radford isomorphism (\ref{eq:Radford}) for the category $\C=\Rep(H)$ is given by
\begin{align}\label{eq:radfordIsomorphism}
    \fR_X: \lv\lv X \otimes D \rightarrow  D \otimes  X \rv\rv, \hspace{0.4cm}    \phi_X(x) \otimes c  \mapsto c \otimes \phi_X(\ogH\cdot x) \hspace{0.6cm} \text{for}\; c\in D, \, x\in X.
\end{align}
\end{proposition}
\begin{proof}
Plugging the formula for $g_X$ into $(\ref{eq:radIso1})$, we get that
\begin{equation*}
    \phi_X(x)\ok c \xmapsto{\coev_D\otimes \id}
    1\ok 1\ok \phi_X(c)\ok c \xmapsto{\id\otimes g_X\otimes \id}
    1\ok \phi_X(\ogH\cdot x) \ok 1\ok c \xmapsto{\id\otimes \ev_D} 
    c\ok \phi_X(\ogH\cdot x).
\end{equation*}
Hence, the claim follows.
\end{proof}

In the following, we will often use the identification $\lv\lv X\otimes D \cong \lv\lv X$ and $D\otimes X\rv\rv\cong X\rv\rv$ because $D\cong\kk$ as a vector space. After these identifications, the Radford isomorphism (\ref{eq:radfordIsomorphism}) becomes $\fR_X(\phi_X(x)) = \phi_X(\ogH\cdot x)$. Its inverse is given by $\fR^{-1}_X(\phi_X(x)) = \phi_X(g_H\cdot x)$.


\subsubsection{Serre functor}\label{subsubsec:hopfSerre}
By Theorem~\ref{thm:NakSerre}, the relative Serre functor satisfies that $\dS \cong D\tr \dN$ as a left $\C$-module functor. Since, the Serre functor is unique up to isomorphism, we take the above as the definition of it. Then we get the following result.

\begin{theorem}
A relative Serre functor of $\Rep(A)$ is given by $\dS(M)={}_{\nu'}(M)$ where 
\begin{equation}\label{eq:nu'}
    \nu'(a)= \langle \alpha_H,S(a_{(-1)}) \rangle \nu(a_{(0)}).
\end{equation} 
The twisted left $\Rep(H)$-module structure $\fs^l_{X,M}: {}_{\nu'}(X\tr M) \xrightarrow{\sim} \lv\lv X\tr {}_{\nu'}M $ of $\dS$ is given by
\begin{equation}\label{eq:fsl}
    \fs(x\ok m) = \langle\lambda, a^i_{(0)} \rangle \phi_X(\gH S^{-1}(a^i_{(-1)})\cdot x) \ok \nu(b_i)\cdot m
\end{equation}  
\end{theorem}
\begin{proof}
The formula $(\ref{eq:nu'})$ follows from the formula of the Nakayama functor $(\ref{eq:nu})$ and the description of $D$ in Lemma~\ref{lem:D}(b). The formula $(\ref{eq:fsl})$ follows by composing the twisted left $\Rep(H)$-module structure $\fn^l_{X,M}$ of $\dN_{\Rep(A)}$ given in $(\ref{eq:fnl})$ and the inverse $\fR_X^{-1}$ of the Radford isomorphism $(\ref{eq:radfordIsomorphism})$.
\end{proof}

\begin{remark}
A different formula for the relative Serre functor was also provided in \cite{shimizu2019relative}. 
\end{remark}


\subsection{Unimodular structures on $\Rep(A)$}\label{subsec:hopfUnimodular}
In this section, we use the description of the Nakayama functor and the relative Serre functor of $\Rep(A)$ provided in Sections~\ref{subsubsec:hopfNakayama} and \ref{subsubsec:hopfSerre}, respectively, to describe the $\C$-module functor $\dS_{\Rep(A)}\dN_{\Rep(A)}$ in Section~\ref{subsubsec:hopfUniFunctor}. Using this, we characterize unimodular structures on $\Rep(A)$ in terms of unimodular elements of $A$ in Section~\ref{subsubsec:uniElement}. As the definition of unimodular elements is involved, in Section~\ref{subsubsec:grouplikeCoint} we discuss the simpler case when $A$ admits a grouplike cointegral. In this case, the definition of unimodular elements is much simpler. 

\subsubsection{Description of the functor $\dS_{\Rep(A)}\dN_{\Rep(A)}$}\label{subsubsec:hopfUniFunctor}

\begin{theorem}\label{thm:uniFunctor}
For $M\in \Rep(A)$, we have that $\dS \, \dN(M) = {}_{\tnu}M \in \Rep(A)$ where
\begin{equation}\label{eq:uniFunctor}
    \tnu(a)=\langle\alpha_H,S(a_{-1}) \rangle \nu^2(a_0).
\end{equation}  
The left $\C$-module structure of $\dS \, \dN$ is given by
\begin{equation}\label{eq:UniModuleStrH}
    \fd_{X,M}(x\tr m) = (\Im_H\cdot x)  \tr (\Im_L\cdot m), \; \text{where}
\end{equation}  
\begin{equation}\label{eq:NSmoduleStrucutre}
    \Im = \Im_H\otimes \Im_A:= \langle \lambda_A,a^i_0\rangle \langle \lambda_A, a^j_0\rangle \;g_H S^{-3}(a^j_{-1})S^{-1}(a^i_{-1})
    \ok
    \nu(b_j b_i)
    \in H\ok A.
\end{equation}
\end{theorem}
\begin{proof}
By the descriptions of the Nakayama functor (\ref{eq:nu}) and the relative Serre functor (\ref{eq:nu'}), it is clear that as a functor, $\dS\,\dN$ is given by (\ref{eq:uniFunctor}). Further, using the the twisted left $\C$-module structures of the relative Serre functor (\ref{eq:fsl}) and the Nakayama functor (\ref{eq:fnl}) , we get that, $\dS \, \dN: \M\rightarrow \M$ is a left $\C$-module functor with the left $\C$-module structure $\fd_{X,M}$ given by the following composition.   
\begin{equation*}
    \fd_{X,M}\,:\,\dS\dN(X\tr M) \xrightarrow{\dS(\fn_{X,M})}  \dS( X\rv\rv \tr \dN(M) ) \xrightarrow{\fs_{X\rv\rv, \dN(M)}}  \lv\lv X\rv\rv \tr\dS\dN(M) = X\tr \dS\dN(M).
\end{equation*}      
The following calculation yields an explicit description of the left $\C$-module structure of $\dS\,\dN$.
\begin{align*}
    x\ok  m  \;\;\;\;\xmapsto{\dS(\fn_{X,M})} \;\;
    & \langle \lambda,a^i_0\rangle \phi_X(S^{-1}(a^i_{-1})\cdot x) \ok \nu(b_i)\cdot m  
    \\[-0.12em]
    \xmapsto{\fs_{X\rv\rv, \dN(M)}} 
    & \langle \lambda_A,a^j_0\rangle \langle \lambda_A,a^i_0\rangle \phi_{X\rv\rv}(g_H S^{-1}(a^j_{-1})\cdot \phi_X(S^{-1}(a^i_{-1})\cdot x)) \ok \nu(b_j)\nu(b_i)\cdot m  
    \\[-0.12em]
    \stackrel{(\spadesuit)}{=} \;\;\;\;\;\; & \langle \lambda_A,a^j_0\rangle \langle \lambda_A,a^i_0\rangle
    \phi_{X\rv\rv}\circ \phi_X (S^{-2}(g_H S^{-1}(a^j_{-1}))S^{-1}(a^i_{-1})\cdot x) 
    \ok \nu(b_j)\nu(b_i)\cdot m 
    \\[-0.12em]
    \stackrel{(\diamondsuit)}{=} \;\;\;\;\;\; & 
    \langle \lambda_A,a^j_0\rangle \langle \lambda_A,a^i_0\rangle
    S^{-2}(g_H S^{-1}(a^j_{-1}))S^{-1}(a^i_{-1})\cdot x
    \ok \nu(b_j b_i)\cdot m 
    \\[-0.12em]
    \stackrel{(\heartsuit)}{=} \;\;\;\;\;\; & 
    \langle \lambda_A,a^j_0\rangle \langle \lambda_A,a^i_0\rangle
    g_H S^{-3}( a^j_{-1}) S^{-1}(a^i_{-1})\cdot x
    \ok \nu(b_j b_i)\cdot m 
\end{align*}
Here $(\lambda_A,\{a^i\},\{b_i\})$ and $(\lambda_A,\{a^j\}, \{b_j\})$ are Frobenius systems of $A$.
The equality $(\spadesuit)$ holds because the left action of $H$ on $ X\rv\rv$ is given $h\cdot \phi_X(x) = \phi_X(S^{-2}(h)\cdot x)$. The equality $(\diamondsuit)$ holds because we identify $X$ and $\lv\lv X\rv\rv$ via the map $\phi_{X\rv\rv}\circ \phi_X$ and  $\nu$ is an algebra map. Lastly, $(\heartsuit)$ holds because $g_H$ is grouplike. 
From the above computation, it follows that the left $\C$-module structure of $\dS\;\dN$ is as described by equations (\ref{eq:UniModuleStrH}) and (\ref{eq:NSmoduleStrucutre}).
\end{proof}

\subsubsection{Unimodular elements in exact $H$-comodule algebras}\label{subsubsec:uniElement}
From Definition~\ref{defn:uniModCat}, recall that $\Rep(A)$ is a unimodular $\Rep(H)$-module category if and only if there exists a $\C$-module natural isomorphism $\fu: \id_{\M} \rightarrow \dS\;\dN$. Next, we will use Theorem~\ref{thm:uniFunctor}, to characterize such natural isomorphisms using certain invertible elements of the algebra $A$. To this end, consider the following notion.

\begin{definition}\label{defn:uniElement}
Let $A$ be an exact, left $H$-comodule algebra. A \textit{unimodular element} of $A$ is an invertible element $\tg\in A$ satisfying the following two relations:
\begin{align}
    \tg a\tg^{-1} &= \tnu(a) \;\;\;\;\; (\forall \hspace{.1cm} a\in A) \label{eq:uniElement1}\\
    1_H \otimes \tg &= \Im\cdot \delta(\tg) \label{eq:uniElement2}
\end{align}
Here $\delta$ is comodule structure of $A$. For the definition of $\tnu$ and $\Im$, see (\ref{eq:uniFunctor}) and (\ref{eq:UniModuleStrH}), respectively. 
\end{definition}

Using this, we obtain the following result.

\begin{theorem}\label{thm:uniclassification}
Let $A$ be an exact, left $H$-comodule algebra. Then, unimodular structures on the $\Rep(H)$-module category $\Rep(A)$ are in bijection with unimodular elements of $A$.
\end{theorem}
\begin{proof}
Suppose that we have a unimodular structure $\fu:\id_{\M}\, \rightarrow\,  \dS\,\dN \stackrel{(\textnormal{\ref{eq:uniFunctor}})} {=}{}_{\tnu}(-)$. Then, we get the element  $\tg = \fu_A(1_A)\in A$. By naturality of $\fu$, it follows that $\fu_X(x)=\tg\cdot x$ for all $x\in X$ and $X\in \Rep(A)$. As $\fu_A$ is an isomorphism, $\tg$ is an invertible element in $A$.
\begin{itemize}
    \item As $\fu_A = \tg\cdot (-):A\rightarrow {}_{\tnu}A$ is a map of left $A$-modules, condition (\ref{eq:uniElement1}) is satisfied. 
    \item As $\fu$ is a $\C$-module natural isomorphism, the diagram (\ref{eq:unimodularModule}) commutes. Using $\fu_X(x) = \tg\cdot x$ and that the $\C$-module structure of $\dS\,\dN$ is given by $\Im$ (\ref{eq:UniModuleStrH}), we get that (\ref{eq:uniElement2}) is satisfied.
\end{itemize}
Thus, $\tg$ is a unimodular element of $A$. 

Conversely, given a unimodular element $\tg$ of $A$, we define the natural isomorphism
\[\fu =\{ \fu_M:M\rightarrow \dS\;\dN(M), \hspace{0.25cm} m\mapsto \tg\cdot m \}_{M\in\Rep(A)}.\]
Then, repeating the above arguments backwards, we get that $\fu=\{ \fu_M \}$ is a unimodular structure on $\Rep(A)$. 
\end{proof}


\subsubsection{Grouplike cointegrals on comodule algebras}\label{subsubsec:grouplikeCoint}
Consider the following notion. 

\begin{definition}\cite{kasprzak2018generalized}\label{defn:grouplike}
Let $H$ be a Hopf algebra and $A$ a left $H$-comodule algebra. A \textit{grouplike cointegral} on $A$ is a pair $(g,\lambda)$ consisting of a grouplike element $g\in H$ and a linear form $\lambda:A\rightarrow \kk$ such that the equation 
\begin{equation}\label{eq:grouplike}
    a_{(-1)} \langle \lambda,a_{(0)} \rangle = \langle \lambda,a \rangle g
\end{equation}
holds for all elements $a\in A$. In this situation, $\lambda$ is called a \textit{$g$-cointegral} on $A$. If $\lambda$ is a Frobenius form on $A$, the $g$-cointegral $\lambda$ is called \textit{non-degenerate}.
\end{definition}

Next, we see the the $\C$-module structure of the functor $\dS\,\dN$ simplifies when we have a grouplike cointegral on $A$.
\begin{theorem}
Recall the element $\Im$ (\ref{eq:NSmoduleStrucutre}) and the left $\C$-module structure of $D_{\Rex_{\C}(\M)}=\dS\;\dN$ (\ref{eq:UniModuleStrH}). 
If the Frobenius form $\lambda_A$ of $A$ is a $g_A$-grouplike integral for some grouplike element $g_A\in H$, then we have that
\begin{equation}
    \Im =  g_A^{-2} g_H \otimes 1_A.
\end{equation}
Thus, for $X\in \C$ and $M\in \M$, the left $\C$-module structure of $\dS\,\dN$ given by:
\[ \dS\;\dN(X\tr M) \rightarrow X\tr \dS\;\dN(M), \hspace{1cm} x\tr m  \mapsto  \phi_X( g^{-2}_A g_H \cdot x )\tr m. \]
\end{theorem}
\begin{proof}
Observe that 
\begin{align*}
\Im & \stackrel{(\textnormal{\ref{eq:NSmoduleStrucutre}})}{=} \langle \lambda_A,a^i_0\rangle \langle \lambda_A, a^j_0\rangle \;g_H S^{-3}(a^j_{-1})S^{-1}(a^i_{-1})
\ok
\nu(b_j b_i)
\\[-0.12cm]
& \stackrel{(\textnormal{\ref{eq:grouplike}})}{=} 
\langle \lambda_A,a^i\rangle \langle \lambda_A, a^j\rangle \;g_H S^{-3}(g_A)S^{-1}(g_A)
 \ok
 \nu(b_j b_i)
\\[-0.12cm]
\;\; & \; \stackrel{(\textnormal{\ref{eq:FrobSystem}})}{=}  
g_H S^{-3}(g_A)S^{-1}(g_A)
\ok
\nu(1_A)   \\[-0.12cm]
& \; \, \stackrel{(\spadesuit)}{=}  g_H g^{-1}_A g^{-1}_A \otimes 1_A \\[-0.12cm]
& \;\,\stackrel{(\diamondsuit)}{=}  g^{-2}_A g_H \otimes 1_A.
\end{align*}
Here, the equality $(\spadesuit)$ holds because $g_A$ is grouplike element of $H$ and $\nu$ is an algebra map. The equality $(\diamondsuit)$ holds because $g_H$ commutes with all grouplike elements of $H$. 
\end{proof}

When the Frobenius form on the $H$-comodule algebra under consideration is a grouplike cointegral, Theorem~\ref{thm:uniclassification} simplifies a lot and we recover \cite[Corollary~7.10]{shimizu2022Nakayama}.

\begin{corollary}\label{cor:grouplikecointegral}
If the Frobenius form $\lambda_A$ of $A$ is a $g_A$-grouplike cointegral for some $g_A\in H$, then, the unimodular structures on $\Rep(A)$ are in bijection with invertible elements $\tg\in A$ satisfying:
\begin{equation}\label{eq:uniElementGrouplike}
    \tg a\tg^{-1}=\tnu(a)=\langle\alpha_H,S(a_{-1}) \rangle \nu^2(a_0), \hspace{1cm} g_H^{-1} g_A^2\otimes\tg =\delta(\tg) \;\;\;\;\;\;\;\;\;\;\;\;\; \forall \; a\in A.
\end{equation}
\qed
\end{corollary}

The category $\Vect=\Rep(\kk)$ is an exact left $H$-module category as $\kk$ is an exact left $H$-comodule algebra with the $H$-coaction given by $1\mapsto 1_H\otimes 1$. 
Then, by applying Corollary~\ref{cor:grouplikecointegral} to the $H$-comodule algebra $\kk$, we obtain $\kk$ admits a unimodular structure if and only if $g_H=1_H$, that is, $H^*$ is unimodular. This is consistent with our observation in Example~\ref{ex:uniModCat}(iii).


\subsection{Taft algebras example}\label{subsec:Taft}
In this section, we study the case of $\C=\Rep(H)$ for $H$ being the Taft algebra $T(\omega)$. Let $\kk$ be an algebraically closed field of characteristic $0$.
To define Taft algebras, fix an integer $N>1$ and a primitive $N$-th root of unity $\omega\in\kk$. Then $T(\omega)$ is defined as the $\kk$-algebra generated by $g$ and $x$ subject to the relations 
\begin{equation}\label{eq:TaftRelations}
    x^N=0, \;\; g^N=1 \;\;\; \text{and} \;\;\; gx= \omega xg.
\end{equation}
Equipped with the following comultiplication and antipode maps, $T(\omega)$ becomes a Hopf algebra. 
\begin{equation}\label{eq:TaftComult}
    \Delta(g) = g\otimes g ,\;\;\Delta(x)=x\otimes 1+ g\otimes x\;\;\; \text{and} \;\;\;  S(g)=g^{-1}, \;\; S(x) = -g^{-1}x.
\end{equation}
From this, we get the following information.

\begin{itemize}
    \item The element $\Lambda^l=\sum_{i=0}^{N-1} g^i x^{N-1}$ is a non-zero left integral of $T(\omega)$. The distinguished character of $T(\omega)$ is given by $\alpha_{T(\omega)}(g)=\omega$ and $\alpha_{T(\omega)}(x)=0$.
    
    \item The functional $\lambda_{T(\omega)}:T(\omega)\rightarrow \kk$ given by $\lambda_{T(\omega)}(x^rg^s) = \delta_{r,N-1} \delta_{s,0}$ for $r,s=0,\ldots,N-1$ is a right cointegral of $T(\omega)$. The distinguished grouplike element of $T(\omega)$ is given by $g_{T(\omega)}=g^{-1}$.
\end{itemize}

As $g_{T(\omega)}\neq 1_H$, we have that $T(\omega)^*$ is not unimodular. Thus, by Example~\ref{ex:uniModCat}, we have that $\Vect$ is not a unimodular $\Rep(T(\omega))$-module category. In fact, as we will see below, $\Rep(T(\omega))$ does not admit a unimodular module category.

Indecomposable, left, exact $T(\omega)$-comodule algebras (or equivalently, indecomposable, exact $\Rep(T(\omega))$-module categories) were classified by Mombelli \cite{mombelli2010module}. Shimizu \cite[\S5.1]{shimizu2019relative} showed that these comodule algebras admit grouplike cointegrals and described them explicitly. Using these results and Corollary~\ref{cor:grouplikecointegral}, we obtain the following result on non-existence of unimodular module categories.

\begin{theorem}\label{thm:Taft}
The Taft algebra $T(\omega)$ does not admit a unimodular comodule algebra. In other words, the category $\Rep(T(\omega))$ does not admit a unimodular module category. 
\end{theorem}

\begin{proof}
Choose a divisor $d|N$ and an element $\xi\in\kk$. Set $m=N/d$ and consider the following algebras:
\begin{itemize}
    \item[\upshape{(a)}] $A_0(d) = \kk \langle G|G^d=1 \rangle.$
    \item[\upshape{(b)}] $A_1(d,\xi) = \kk \langle G,X|G^d=1, X^N=\xi , GX=\omega^m XG\rangle.$
\end{itemize}
They are $T(\omega)$-comodule algebras with the coaction determined by 
\begin{equation}\label{eq:TaftCoaction}
    \delta(G)= g^m\otimes G \;\;, \delta(X)=x\otimes 1 + g\otimes X.
\end{equation}
By \cite[Proposition~8.3]{mombelli2010module}, every indecomposable exact left module category over $\Rep(T(\omega))$ is equivalent to $\Rep(A)$ where $A$ is one of the comodule algebras listed above. Next, we recall the grouplike cointegrals for these comodule algebras and use them to show that the module categories $\Rep(A)$ are not unimodular.

\medskip

\noindent
\textit{The comodule algebras $A_0(d)$}:
Define the linear map $\lambda_A:
A_0(d) \rightarrow \kk$ by $\lambda_A(G^r)=\delta_{0,r}$ (for $r\in\zZ/d\zZ$). Then, by \cite[Section~5.1.1]{shimizu2019relative}, $\lambda_A$ is a $g_A$-grouplike cointegral with $g_A =1$ on $A_0(d)$. Further, $\lambda_A$ is a Frobenius form on $A_0(d)$.
The Nakayama automorphism with respect to $\lambda_A$ is $\nu=\id_{A_0(d)}$. Next, we calculate the automorphism $\tnu$ defined in $(\ref{eq:uniFunctor})$.
\begin{equation*}
    \tnu(G) \stackrel{}{=}  \langle\alpha_{T(\omega)}, S(g^{m}) \rangle \nu^2(G) = \langle\alpha_{T(\omega)}, g^{-m} \rangle G = \omega^{-m}G 
\end{equation*}
Plugging this, $g_{T(\omega)}=g^{-1}$ and $g_A=1$ into Corollary~\ref{cor:grouplikecointegral}, we get that $A$ admits a unimodular structure if and only if there exists an element $\tilde{G} \in A_0(d)$ such that 
\begin{equation}\label{eq:taft1}
    \tG G^r= \tnu(G^r)\tG = \omega^{-mr}G^r \tG = \omega^{-mr} \tG G^r \;\; \text{and} \;\;  g \ok \tG = \delta(\tG) \;\;\;\text{hold}\; \forall\; r\in \zZ/d\zZ.
\end{equation} 
The first condition of (\ref{eq:taft1}) is satisfied if and only if $\omega^{-mr}=1$ for all $r$. This is satisfied only if $m=N$ and $d=1$. With this choice of $d$, $A_0(d)\cong \kk$. Then, the second condition of (\ref{eq:taft1}) is not satisfied for any $\tG$. Hence, the module categories $\Rep(A_0(d))$ are not unimodular for any $d$.

\medskip

\noindent
\textit{The comodule algebras $A_1(d,\xi)$}:
Observe that the set $\{X^rG^s | r = 0,\ldots, N - 1; s = 0, \ldots, d - 1\}$ is a basis of $A_1(d,\xi)$.
By \cite[Section~5.1.2]{shimizu2019relative}, the linear map $\lambda_A:
A_1(d,\xi)\rightarrow \kk$ given by 
\begin{equation*}
    \lambda_A(X^r G^s)=\delta_{r,N-1}\delta_{s,0}\;\;\;\;\;\;\; (\text{for}\;  r\in \{ 0,\ldots ,N-1\}\; \text{and}\; s\in\zZ/d\zZ)
\end{equation*} is a $g^{-1}$-cointegral and a Frobenius form on $A_1(d,\xi)$. The Nakayama automorphism $\nu$ with respect to $\lambda_A$ is given by $\nu(X)=X$, $\nu(G)=\omega^m G$. 
Using the formula $(\ref{eq:uniFunctor})$, we get that
\begin{align*}
    \Scale[0.98]{
    \tnu(G) = \langle\alpha_{T(\omega)}, S(g^{m}) \rangle \nu^2(G) 
    = \omega^{m}G ,
    \;\;\;\;
    \tnu(X)  = \langle\alpha_{T(\omega)}, S(x) \rangle \nu^2(1) + \langle\alpha_{T(\omega)}, S(g) \rangle \nu^2(X)
    = \omega^{-1}X.
    }
\end{align*}
Hence, $\tnu(X^rG^s) = \omega^{ms -r}X^rG^s$.
For $A_1(d,\xi)$ to be unimodular, by the first condition of (\ref{eq:uniElementGrouplike}), we want $\tnu$ to be an inner automorphism. As in \cite[Section~5.1.2]{shimizu2019relative}, by a case-by-case analysis, we show that $\tnu$ is not an inner automorphism.
\begin{itemize}
    \item \underline{$\xi=0$, $ d> 1$}: Consider the non-zero algebra map $\epsilon:A_1(d,\xi)\rightarrow \kk$ given by $\epsilon(X)=0$ and $\epsilon(G)=1$. Then, $\epsilon\circ \tnu (G) = \omega^m \neq 1$ as $m<N$. Thus, $\epsilon\circ \tnu (G) \neq \epsilon(G)$, and so $\tnu$ is can not be an inner automorphism.
    
    \item \underline{$\xi=0$, $ d= 1$}: In this case, $A_1(d,\xi) = \kk\langle X|X^N=0 \rangle $. As $A_1(d,\xi)$ is commutative, any inner automorphism will just be the identity map. However, $\tnu(X)=\omega^{-1} X \neq X$. Hence, $\tnu$ is not inner.
    
    \item \underline{$\xi\neq 0$, $d<N$}: Fix a $N$-th root $\zeta$ of $\xi$. Define the left $A_1(d,\xi)$-module $V$ as follows. A basis of $V$ is given by $\{v_i\}_{i\in \zZ/d\zZ}$ and the action is given by 
    \[ X\cdot v_i = \zeta v_{i+1}, \;\;\;\;\; G\cdot v_i= \omega^{mi} v_i \;\;\;\;\;\;\;\;\;\;\;\;\;\;\; (i\in \zZ/d\zZ). \]
    Also, consider the $\tnu$-twisted module ${}_{\tnu}V$. Then $X^d$ acts on $V$ and ${}_{\tnu}V$ as scalars $\zeta^d$ and $\omega^{-d}\zeta^d$, respectively. As $d<N$, we have that $\omega^{-d}\neq 1$. Thus, $V$ and ${}_{\tnu}V$ are not isomorphic as left $A_1(d,\xi)$-modules. Hence, $\tnu$ can not be an inner automorphism. 

    \item \underline{$\xi\neq 0$, $d=N$}: Consider the non-zero algebra map $\epsilon: A_1(N,\xi)\rightarrow \kk$ given by $\epsilon(G)=1$ and $\epsilon(X)=\zeta$ for $\zeta$ a $N$-th root of $\xi$. Then, $\epsilon\circ \tnu (X) = \omega^{-1} \zeta \neq \zeta$. Thus, $\epsilon\circ \tnu (X) \neq \epsilon(X)$, and so $\tnu$ is not an inner automorphism.
\end{itemize}
Hence, the module categories $\Rep(A_1(d,\xi))$ are not unimodular for any $d$ and $\xi$.
\end{proof}


\subsection{Further remarks and questions}\label{subsec:remarks} 
We end this section by listing some remarks and directions for further investigation.

\smallskip

First, we show that Theorem~\ref{thm:uniclassification} answers a question of Shimizu \cite[Question~7.15]{shimizu2022Nakayama}.
Take $H$ to be a finite dimensional Hopf algebra over a field $\kk$.
Let $\D=\Corep(H)$ denote the category of left $H$-comodules. Consider a left $H$-comodule algebra $A$. Then, $A$ is nothing but an algebra object in the category $\D$. Furthermore, the category of $A$-bimodules in the $\D$, denoted ${}_A\D_A$, monoidally equivalent to $\Rex_{\C}(\M)$ for $\C=\Rep(H)$ and $\M=\Rep(A)$. Moreover, when $A$ is exact, both these are multitensor categories. So, one can ask when they are unimodular.
In \cite{shimizu2022Nakayama}, the unimodularity of ${}_A\D_A$ was studied under the assumption that the algebra $A$ admits a grouplike cointegral, see Definition~\ref{defn:grouplike}. Further, in \cite[\S 7.5]{shimizu2022Nakayama}, an example of an exact left $H$-comodule algebra that does not admit a grouplike cointegral was provided and it was asked if there is an easy criterion for determining the unimodularity of ${}_A\D_A$ in the general case.

By Lemma~\ref{lem:uniDefn}, the multitensor category $\Rex_{\C}(\M)$ (or ${}_A\D_A$) is unimodular if and only if $\M$ is a unimodular $\C$-module category. Thus, the following Corollary of Theorem~\ref{thm:uniclassification} provides an answer to Shimizu's question.

\begin{corollary}\label{cor:answerShimizu}
For $\C=\Rep(H)$ and $\M=\Rep(A)$, the category $\Rex_{\C}(\M)$ is unimodular if and only if $A$ admits a unimodular element. \qed
\end{corollary}

While Corollary~\ref{cor:answerShimizu} answers Shimizu's question, it is not an easy criterion in general. For instance, see Section~\ref{subsec:Taft} for an example of a computation. This inspires the following discussion.

\smallskip

A finite dimensional Hopf algebra is unimodular if and only if it admits a two sided integral. To define the integrals, the counit $\varepsilon$, which is an algebra map from $H$ to $\kk$, is crucially used. It would be interesting to find a similar characterization of unimodularity of exact $H$-comodule algebras. However, we do not know whether such algebras $A$ admits an algebra map to $\kk$. This raises the following question.

\begin{question}\label{question:unimodularity}
Let $A$ be an exact left $H$-comodule algebra. Is there a way to define left and right integrals for $A$. If so, can the integrals be used to characterize the unimodularity of the exact left $\Rep(H)$-module category $\Rep(A)$? 
\end{question}

\smallskip

Next, two finite tensor categories $\C,\D$ are called \textit{categorically Morita equivalent} if there exists an indecomposable exact left $\C$-module category $\M$ such that $\D^{\rev}\cong \Rex_{\C}(\M)$ as finite tensor categories. It is clear that a tensor category admits a unimodular module category if and only if it is categorically Morita equivalent to a unimodular tensor category. 
Thus, Theorem~\ref{thm:Taft} established that the category $\Rep(T(\omega))$ is not categorically Morita equivalent to a unimodular tensor category. This naturally leads to the following question.
\begin{question}
Find a characterization of finite tensor categories that do not admit a unimodular module category.
\end{question}

\begin{remark}
Let $(\C,\fp)$ be a pivotal finite tensor category. Then, we call $\C$ \textit{trace-spherical} if $\dim(X)=\dim(\lv X)$ holds for all objects $X\in\C$. On the other hand, a pivotal finite tensor category $(\C,\fp)$ is called \textit{(DSPS-)spherical} \cite{douglas2018dualizable} if $\C$ is unimodular and it satisfies that \[  \fp_{X}\circ \fp_{X\rv\rv} = (\fR_X)^{-1}: X\rv\rv \rightarrow \lv\lv X  \; \text{for all} \; X\in\C.\]  
It is known that these two notion of sphericality are not the same. For instance, it was shown in \cite{douglas2018dualizable} that $\Rep(T(\omega))$ is trace-spherical but not DSPS-spherical. 
By Theorem~\ref{thm:Taft}, we obtain that $\Rep(T(\omega))$ is not categorically Morita equivalent to a unimodular tensor category. This, in particular, implies that $\Rep(T(\omega))$ can not be categorically Morita equivalent to a DSPS-spherical tensor category. 
This establishes that the two notions of sphericality are not equivalent even when one considers the weaker notion of categorical Morita equivalence.

\end{remark}


\appendix

\section{Nakayama functor of \texorpdfstring{$\Rep(A)$}{Rep(A)} and its twisted \texorpdfstring{$\Rep(H)$}{Rep(H)}-module structure}\label{app:A}
Let $H$ be a finite-dimensional Hopf algebra and $A$ be a left $H$-comodule algebra. In this appendix we provide proofs of Theorems~\ref{thm:RepANak}, \ref{thm:RepAtwistedLeft} and \ref{thm:RepAtwistedRight} which pertain to the (right) Nakayama functor of the $\Rep(H)$-module category $\Rep(A)$ and its twisted module structure.

\begin{notation}
Consider three $\kk$-vector spaces $M,N,N'$ and a $\kk$-linear map $f:N\rightarrow N'$. We will denote by $f^{\natural}:\Hom(M,N)\rightarrow \Hom(M,N')$, the map $g\mapsto f^{\natural}(g)=f\circ g$. 
Also, for any algebra $A$, we have that $A^*$ is a $A$-bimodule via the actions 
\begin{equation}\label{eq:dualLbimodule}
    \langle a'\rightharpoonup f \leftharpoonup a'',a   \rangle := \langle f,a'' a a'\rangle \hspace{1cm} (a,a',a''\in A, \;f\in A^*). 
\end{equation}
\end{notation}

For this section, we fix $A$ to be an exact left $H$-comodule algebra. Let $\lambda_A$ and $\nu$ denote the Frobenius form and the Nakayama automorphism of $A$, respectively.
We will need the following result.
\begin{lemma}\label{lem:NakLemma}
    Let $A$ be an exact left $H$-comodule algebra. Then, the following results hold.
\begin{enumerate}
    \item[\upshape{(a)}] For $V\in\Vect$ and $M\in \Rep(A)$, the canonical $\Vect$ action is given by $V\btr M=V\ok M$ as a vector space and $a\cdot(v\ok m) = v\ok a\cdot m$.

    \item[\upshape{(b)}] If $A$ is a left $H$-comodule algebra, $A^*$ also becomes a left $H$-comodule with the coaction given by
    $\rho_{A^*}(f):= f_{(-1)}\otimes f_{(0)}\in H\otimes A^*$ 
    where,
    \begin{equation}\label{eq:dualComodule}
        f_{(-1)} \langle f_{(0)}, a\rangle = \langle f, a_{(0)} \rangle S^{-1}(a_{(-1)}) \hspace{1cm} (a\in A, \;f\in A^*).    
    \end{equation}

    \item[\upshape{(c)}] Let $M,\,N$ be left $A$-modules. Then, the map $\psi: N^*\otimes_A M \rightarrow \Hom_A(M,N)^*$ given by $n^*\oA m\mapsto \langle n^*,?(m)\rangle$ is an isomorphism of vector spaces.
    
    \item[\upshape{(d)}] The endofunctors ${}_{(\nu)}(-)$ and $ A^*\otimes_A -$ of the category $\Rep(A)$ are isomorphic via the following natural isomorphisms
    \begin{align}
        \label{eq:alphaNak} \alpha_M:A^*\otimes_A M \rightarrow {}_{\nu}M, & \hspace{1cm} f\otimes_A m\mapsto \langle f,a^i\rangle \nu_A(b_i)\cdot m . \\
        \label{eq:betaNak} \beta_M:{}_{\nu}M \rightarrow A^*\otimes_A M, & \hspace{1cm} m\mapsto \lambda_A\otimes_A m ,
    \end{align}  
\end{enumerate}
\end{lemma}
\begin{proof}
Parts $(a)$ and $(b)$ are straightforward to check. Part $(c)$ is \cite[Lemma~4.1]{shibata2021modified}. Thus, we only prove part $(d)$ below.

It is straightforward to check that $\alpha_M$ and $\beta_M$ are natural in $M$ and are maps of left $A$-modules. Below, we check they are isomorphisms.

\[
\begin{array}{rclcl}
    \beta_M[\alpha_M(f\otimes_L m)] & 
    \stackrel{(\textnormal{\ref{eq:alphaNak}}) }{=}  & 
    \beta_M[\langle f,a^i \rangle \nu(b_i)\cdot m] & 
    \stackrel{(\textnormal{\ref{eq:betaNak}}) }{=} &
    \lambda_A\otimes_A \langle f,a^i\rangle \nu(b_i)\cdot m 
    \\
    & \stackrel{(\diamondsuit)}{=} &\langle f,a^i \rangle \; (\lambda_L \leftharpoonup \nu(b_i)) \otimes_A m & \stackrel{(\textnormal{\ref{eq:dualLbimodule}})}{=} &
    \langle f,a^i\rangle \; \langle\lambda_A,\nu(b_i)\; ?\rangle \otimes_A m 
    \\
    & \stackrel{(\textnormal{\ref{eq:NakayamaAuto}})}{=} & \langle f,a^i\rangle \; \langle\lambda_A, ?\; b_i \rangle \otimes_A m  & \stackrel{(\textnormal{\ref{eq:FrobSystem}})}{=} &
    \langle f,a^i\; ?\rangle \; \langle\lambda_A, b_i \rangle \otimes_A m 
    \\
    & = & \langle f, a^i\langle\lambda_A,b_i\rangle \; ? \rangle \otimes_A m & \stackrel{(\textnormal{\ref{eq:FrobSystem}})}{=} & 
    f\otimes_A m ,\vspace{0.2cm}
    \\
    \alpha_M[\beta_M(m)] & \stackrel{(\textnormal{\ref{eq:betaNak}}) }{=} & 
    \alpha_M[\lambda_A\otimes_A m] & \stackrel{(\textnormal{\ref{eq:alphaNak}}) }{=} &
    \langle \lambda_M,a^i \rangle \nu(b_i)\cdot m 
    \\
    & = &   
    \nu( \langle \lambda_M,a^i \rangle b_i)\cdot m & \stackrel{(\textnormal{\ref{eq:FrobSystem}})}{=}  &
    \nu(1_A)\cdot m \hspace{.75cm} = \hspace{.75cm} m.  
\end{array}
\]
Here the equality $(\diamondsuit)$ holds because $\nu(b_i)\in A$, hence we can move it across the tensor product over $A$. Thus $\alpha$ is natural isomorphism with inverse $\beta$.    
\end{proof}

\changelocaltocdepth{1} 
\subsection{Proof of Theorem~\ref{thm:RepANak}} \label{app:ARepANak}
Recall from (\ref{eq:bi}) that the maps $\bi_{M,N}:\HomA(M,N)^* \; \btr M \rightarrow {}_{\nu}M $ are given by 
\begin{equation*}
    \bi_{M,N}(\xi \ok n) = \langle \xi, \phi^N_M(m^i \oA a^j\cdot n) \rangle \nu(b_j) \cdot m_i  \;\;\;\;\;\;\;\;  (\xi \in \HomA(M,N)^*,\; n\in N)
\end{equation*}

We first show that the component maps $\bi_{M,N}$ admit a right inverse.

\begin{lemma}\label{lem:i1}
The map $\omega :{}_{\nu}M\rightarrow \HomA(M,A)^* \; \btr A$ given by $m\mapsto \langle \lambda_A , ?(m) \rangle \ok 1_A$ satisfies that $\bi_{M,A} \circ \omega = \id_{{}_{\nu}M}$. 
\end{lemma}
\begin{proof}
Observe that  
\begin{align*}
    \bi_{M,A} \circ \omega(m) & =\bi_{M,A}\big( \langle \lambda_A,?(m)\rangle \ok 1_A \big) \\
    & = \langle \langle \lambda_A,?(m)\rangle, \phi_M^A(m^i \oA a^j 1_A) \rangle \;\nu(b_j)\cdot m_i \\
    & =  \langle \lambda_A,\langle m^i,m\rangle a^j \rangle \;\nu(b_j)\cdot m_i  = \langle \lambda_A, a^j \rangle \nu(b_j)\cdot m = m.
\end{align*}
Thus, $\bi_{M,A}$ admits a right inverse. 
\end{proof}

To establish that ${}_{\nu}M$ is equal to the coend $\dN(M)$, we show that the maps $\bi_{M,N}$ are dinatural.
\begin{lemma}\label{lem:i2}
The maps $\bi_{M,N}$ are morphism in $\Rep(A)$, and they are dinatural.
\end{lemma}
\begin{proof}
Observe that 
\begin{align*}
    \bi_{M,N}(a\cdot(\xi\ok n)) 
    & = \bi_{M,N}(\xi\ok a\cdot n) =  \langle \xi, \phi_M^N(m^i\oA a^j a\cdot n)\rangle \nu(b_j)  \cdot m_i \\
    &  = \langle \xi, \phi_M^N(m^i\leftharpoonup a^j a\oA  n)\rangle\nu(b_j)  \cdot m_i  
    \stackrel{(\textnormal{\ref{eq:FrobSystem}})}{=} \langle \xi, \phi_M^N(m^i\oA  n)\rangle  (\nu(b_j)a^j a) \cdot m_i \\
    & \stackrel{(\textnormal{\ref{eq:FrobSystem}})}{=} \langle \xi, \phi_M^N(m^i\oA  n)\rangle  (\nu(a b_j)a^j ) \cdot m_i 
    = \nu(a)\cdot \bi_{M,N}(\xi\ok n) = a\cdot_{\nu}\bi_{M,N}(\xi\ok n) 
\end{align*}
Thus, $\bi_{M,N}$ is a map of left $A$-modules. 
Further, for $f:N\rightarrow N'\in\Rep(A)$ and $\xi\in\HomA(M,N')^*$, we get that
\begin{align*}
    \bi_{M,N'}(\xi\ok f(n)) =  \langle\xi,\phi_M^{N'}(m^i \oA a^j \cdot f(n))\rangle \nu(b_j) \cdot m_i & 
    \stackrel{}{=} \langle\xi, f\circ \phi_M^N(m^i\oA n)\rangle  \nu(b_j)a^j \cdot m_i \\
    & = \bi_{M,N}(\xi\circ f^{\natural} \ok n)
\end{align*}
This proves dinaturality. 
\end{proof}

Next, we show that the pair $({}_{\nu}M,\bi_{M})$ satisfies the universal property of coends.

\begin{lemma}\label{lem:i3}
    The pair $({}_{\nu}M,\bi_{M})$ satisfies the universal property of coends.
\end{lemma}
\begin{proof}
Suppose that there exists an object $X\in \Rep(A)$ along with dinatural maps
\begin{equation*}
    \mathbbm{j}_{M,N}: \HomA(M,N)^* \; \btr N\rightarrow X.
\end{equation*}
We will show that there exists a unique map $\kappa_X:{}_{\nu}M\rightarrow X$ such that $\kappa_X\circ  \bi_{M,N}=\mathbbm{j}_{M,N}$ for all $M\in\M$.
Since $\mathbbm{j}$ is dinatural in $N$, we get that for all $\xi\in\Hom_{\M}(M,N)^*$, $n\in N$ and  $f:N\rightarrow N'$
\begin{equation}
    \mathbbm{j}_{M,N'}(\xi\ok f(n)) = \mathbbm{j}_{M,N}( \xi\circ f^{\natural} \ok n)
\end{equation}
Plugging $f_n:A\rightarrow N$ where $a\mapsto a\cdot n$ in the above equation yields us
\begin{equation}\label{eq:coend1}
    \mathbbm{j}_{M,N}(\xi\ok f_n(a)) = \mathbbm{j}_{M,A}( \xi\circ f_n^{\natural} \ok a)
\end{equation}
With $a=1_A$, LHS is equal to $\mathbbm{j}_{M,N}(\xi\ok n)$. On the other hand,
\begin{align*}
    \omega \circ \bi_{M,N}(\xi\ok n) 
    & = \omega \big( \langle \xi, \phi^N_M(m^i \oA a^j\cdot n) \rangle \nu(b_j) \cdot m_i  \big)
    \\
    & = \langle \xi, \phi^N_M(m^i \oA a^j \cdot n) \rangle
    \langle \lambda_A, ?( \nu(b_j) \cdot m_i) \rangle \ok 1_A 
    \\
    & = \langle \xi, \phi^N_M(m^i \oA a^j \cdot n) \rangle
    \langle \lambda_A, \nu(b_j) ?( m_i) \rangle \ok 1_A 
    \\
    & = \langle \xi, \phi^N_M(m^i \oA a^j \cdot n) \rangle
    \langle \lambda_A, ?( m_i) b_j\rangle \ok 1_A 
    \\
    & = \langle \xi, \phi^N_M(m^i \oA a^j ?(m_i) \cdot n) \rangle
    \langle \lambda_A, b_j\rangle \ok 1_A 
    \\
    & = \langle \xi, \phi^N_M(m^i \oA ?(m_i)\cdot n) \rangle \ok 1_A 
    \\
    & = \langle \xi, \phi^N_M(m^i \oA ?(m_i)\cdot f_n(1_A)) \rangle \ok 1_A 
    \\
    & = \langle \xi, \phi^N_M(m^i \oA f_n(?(m_i))) \rangle \ok 1_A 
    \\
    & = \langle \xi, f_n (\phi^N_M(m^i \oA ?(m_i))) \rangle \ok 1_A 
    \\
    & = \langle \xi \circ f_n^{\natural},\phi^N_M(m^i \oA ?(m_i)) \rangle \ok 1_A
    \\
    & = \langle \xi \circ f_n^{\natural},? \rangle \ok 1_A \; =\; \xi \circ f_n^{\natural} \ok 1_A  
\end{align*}
Plugging the above formula in (\ref{eq:coend1}) at $a=1$, we get that for all $N\in\Rep(A)$
\begin{equation}\label{eq:coend2}
    \mathbbm{j}_{M,N}(\xi\ok n) = \mathbbm{j}_{M,A}( \xi\circ f_n^{\natural} \ok 1_L) = \mathbbm{j}_{M,A} \circ \omega \circ \bi_{M,N}(\xi\ok n) = \kappa \circ \bi_{M,N}(\xi \ok n),
\end{equation}
where $\kappa_X = \mathbbm{j}_{M,A} \circ  \omega $. 
In order to check that $\kappa_X$ is the unique map so that (\ref{eq:coend2}) holds, we plug $N=A$ in (\ref{eq:coend2}) to get $\mathbbm{j}_{M,A} =\kappa_X \circ \bi_{M,A}$. By Lemma~\ref{lem:i1}, $\bi_{M,A}$ admits a right inverse $\omega$. Therefore, $\kappa_X$ is the unique map such that $\mathbbm{j}_{M,N} = \kappa_X \circ \bi_{M,N}$ holds.
\end{proof}

\begin{proof}[Proof of Theorem~\ref{thm:RepANak}]
Together Lemma~\ref{lem:i2} and Lemma~\ref{lem:i3} imply the claim.
\end{proof}


\subsection{Proof of Theorem~\ref{thm:RepAtwistedRight}}\label{app:ARepAtwisted}
Let $A$ be an exact right $H$-comodule algebra. By Lemma~\ref{lem:NakLemma}(d), we know that $A^*\oA M\xrightarrow{\sim} {}_{\nu}M$. Thus, we will employ \cite[Theorem~7.3]{shibata2021modified}, which provided the twisted right $\C$-module structure of the functor $A^*\oA -:\Rep(A)\rightarrow \Rep(A)$, to get the twisted right $\C$-module structure of $\dN(M)={}_{\nu}M$.

Using \cite[Theorem~7.3]{shibata2021modified}, the twisted right $\C$-module structure 
\begin{equation}\label{eq:nakThm3a}
    \Psi_{M,X}:(A^*\oA M)\tl \lv\lv X \rightarrow A^*\oA (M\tl X)   
\end{equation} 
of the functor $A^*\oA -:\Rep(A)\rightarrow \Rep(A)$ is given by 
\begin{equation}
    \Psi_{M,X}[(a^*\oA m)\ok \phi_X(x)] = a^*_{(0)}\oA( m\ok S(a^*_{(1)})\cdot x  )
\end{equation}
Thus, the that twisted right $\C$-module structure of $\dN(M)={}_{\nu}M$ is given by
\begin{align*}
    \ofn^r_{X,M}= (\alpha_{M\tl X } \circ \Psi_{M,X} \circ (\beta_M \tl \id_{\lv\lv X}) ) : \;
    {}_{\nu}M\tl \lv\lv X \;
    & \rightarrow \; {}_{\nu}(M\tl X).     
\end{align*}
The following computation provides an explicit formula for $\ofn^r_{X,M}$.
\begin{align*}
    \ofn^r_{X,M}(m\ok \phi_X(x)) & =  \alpha_{M\tl X } \circ \Psi_{M,X} \circ (\beta_M \tl \id_{\lv\lv X}) \big( m\ok \phi_X(x) \big) \\
    & =  \alpha_{M\tl X } \circ \Psi_{M,X} \big( (\lambda \oA m) \ok \phi_X(x)  \big) \\
    & =  \alpha_{M\tl X }\big(\lambda_{(0)} \oA(m\ok S(\lambda_{(1)})\cdot x )\big)\\
    & =  \langle \lambda_{(0)},a^i \rangle \nu(b_i) \cdot [ m \ok S(\lambda_{(1)})\cdot x] \\
    & =  \langle \lambda,a^i_{(0)} \rangle \nu(b_i) \cdot [ m \ok S (S (a^i_{1}))\cdot x] 
\end{align*}

\smallskip

One can check that the inverse of the map (\ref{eq:nakThm3a}) is given by 
\begin{equation*}
    \Psi^{-1}_{M,X}[a^* \oA (m \ok x) ] = (a^*_{(0)} \oA m) \ok \phi_X(a^*_{(1)}\cdot x)
\end{equation*}
Now again using the isomorphism $A^*\oA M\xrightarrow{\cong} {}_{\nu}M$ from Lemma~\ref{lem:NakLemma}(d), we get that the maps $\fn^r_{X,M}:  {}_{\nu}(M\tl X) \rightarrow {}_{\nu}M\tl \lv\lv X$ is given by the following map.
\begin{align*}
    \fn^r_{X,M}(m\ok x) & =  (\alpha_M \tl \id_{\lv\lv X}) \circ \Psi^{-1}_{M,X} \circ \beta_{M\tl X} \big( m\ok x \big)\\
    & =  (\alpha_M \tl \id_{\lv\lv X}) \circ \Psi^{-1}_{M,X} \big( \lambda \oA(m\ok x) \big) \\
    & =  (\alpha_M \tl \id_{\lv\lv X}) \big( (\lambda_{(0)} \oA m) \ok \phi_X(\lambda_{(1)} \cdot x  ) \big) \\
    & =  \langle \lambda_{(0)},a^i\rangle \nu(b_i) \cdot m \ok \phi_X(\lambda_{(1)} \cdot x) \\ 
    & =  \langle \lambda,a^i_{(0)} \rangle \nu(b_i) \cdot m \ok \phi_X(S(a^i_{(1)}) \cdot x) 
\end{align*}
Hence, the proof is finished. \qed

\bibliographystyle{alpha}
\bibliography{references}

\end{document}

%% file: macros.tex
\makeatletter
\@namedef{subjclassname@2020}{%
  \textup{2020} Mathematics Subject Classification}
\makeatother

\allowdisplaybreaks

\setuldepth{Berlin} 

\theoremstyle{plain}
\newtheorem{theorem}{Theorem}[section]
\newtheorem{lemma}[theorem]{Lemma}

\newtheorem{proposition}[theorem]{Proposition}

\newtheorem{corollary}[theorem]{Corollary}

\newtheorem*{theorem*}{Theorem}

\newtheoremstyle{named}{8pt}{8pt}{\itshape}{}{\bfseries}{.}{.5em}{#1 #3}
\theoremstyle{named}
\newtheorem*{namedtheorem}{Theorem}

\newtheorem*{namedcorollary}{Corollary}

\theoremstyle{definition}
\newtheorem{definition}[theorem]{Definition}
  
\newtheorem{remark}[theorem]{Remark}
\newtheorem{example}[theorem]{Example}

\newtheorem{notation}[theorem]{Notation}

\newtheorem{question}[theorem]{Question}
\newtheorem*{question*}{Question}

\numberwithin{equation}{section}

\makeatletter              
\let\c@equation\c@theorem  
\makeatother

\DeclareMathOperator{\End}{End}

\newcommand{\C}{\mathcal{C}}

\newcommand{\D}{\mathcal{D}}
\newcommand{\E}{\mathcal{E}}
\newcommand{\Z}{\mathcal{Z}}
\newcommand{\M}{\mathcal{M}}
\newcommand{\N}{\mathcal{N}}

\newcommand{\dS}{\mathbb{S}}

\newcommand{\dN}{\mathds{N}}

\newcommand{\fp}{\mathfrak{p}}
\newcommand{\fq}{\mathfrak{q}}
\newcommand{\fs}{\mathfrak{s}}
\newcommand{\fn}{\mathfrak{n}}
\newcommand{\fu}{\mathfrak{u}}
\newcommand{\fd}{\mathfrak{d}}

\newcommand{\fR}{\mathfrak{R}}

\newcommand{\ofn}{\overline{\mathfrak{n}}}
\newcommand{\fnr}{\mathfrak{n}^r}
\newcommand{\fnl}{\mathfrak{n}^l}

\newcommand{\bi}{\mathbbm{i}}

\newcommand{\zZ}{\mathbb{Z}}

\newcommand{\kk}{\Bbbk}
\newcommand{\kkx}{\kk^{\times}}
\newcommand{\ok}{\otimes_{\kk}}

\newcommand{\oA}{\otimes_A}
\newcommand{\id}{\textnormal{\textsf{Id}}}
\newcommand{\Rep}{\textnormal{\textsf{Rep}}}
\newcommand{\Corep}{\textnormal{\textsf{Corep}}}
\newcommand{\lv}{\hspace{.025cm}{}^{\textnormal{\Scale[0.8]{\vee}}} \hspace{-.025cm}}
\newcommand{\rv}{\hspace{.015cm}{}^{\textnormal{\Scale[0.8]{\vee}}} \hspace{-.03cm}}

\newcommand{\ev}{\textnormal{ev}}
\newcommand{\coev}{\textnormal{coev}}

\newcommand{\Vect}{\mathsf{Vec}}
\newcommand{\Hom}{\textnormal{Hom}}
\newcommand{\Homk}{\textnormal{Hom}_{\kk}}

\newcommand{\HomA}{\textnormal{Hom}_A}

\newcommand{\unit}{\mathds{1}}

\newcommand{\uHom}{\underline{\Hom}}
\newcommand{\uNat}{\textnormal{\underline{Nat}}}

\newcommand{\Frob}{{\sf Frob}}

\newcommand{\Rex}{{\sf Rex}}

\newcommand{\ra}{{\sf ra}}
\newcommand{\rra}{{\sf rra}}
\newcommand{\la}{{\sf la}}
\newcommand{\lla}{{\sf lla}}

\newcommand{\tr}{\triangleright}

\newcommand{\btr}{\raisebox{0.15ex}{\Scale[0.85]{\blacktriangleright}}}

\newcommand{\tl}{\triangleleft}

\newcommand{\rev}{\textnormal{rev}}
\newcommand{\op}{\textnormal{op}}

\newcommand{\mir}{\textnormal{mir}}

\definecolor{darkgreen}{RGB}{55,138,0}
\definecolor{burntorange}{RGB}{180,85,0}
\definecolor{navyblue}{RGB}{18,40,180}
\definecolor{cyan(process)}{rgb}{0.0, 0.6, 1.0}

\newcommand{\tev}{\widetilde{\ev}}

\newcommand{\ogH}{\overline{g_H}}
\newcommand{\ostar}{\overline{\star}}
\newcommand{\odN}{\overline{\dN}}
\newcommand{\onu}{\overline{\nu}}
\newcommand{\gH}{g_H}
\newcommand{\tnu}{\widetilde{\nu}}
\newcommand{\tg}{\widetilde{g}}

\newcommand{\tG}{\widetilde{G}}